\documentclass[11pt]{article}
\usepackage{bookmark}
\usepackage{enumitem}
\usepackage[margin=1in]{geometry}
\usepackage{import}

\usepackage{mathtools}
\usepackage{amssymb}
\usepackage{mathrsfs}
\usepackage{amsthm}
\usepackage{xparse}
\usepackage{import}
\usepackage[normalem]{ulem}

\makeatletter
\@ifpackageloaded{unicode-math}{%
	\newcommand{\mymathbold}{\symbf}%
}{%
	\usepackage{bm}%
	\newcommand{\mymathbold}{\bm}%
}
\makeatother

\newcommand{\scrbar}[1]{\overline{\mathcal{#1}}}

\ExplSyntaxOn

\cs_new_protected:Nn \bamboo_define:nnnnnN
{
	\cs_new_protected:cpx { #3 #1 #4 } { \exp_not:N #5{#6{#2}} }
}

\int_step_inline:nnn { `A } { `Z }
{
	\bamboo_define:nnnnnN
	{ \char_generate:nn { #1 } { 11 } }
	{ \char_generate:nn { #1 } { 11 } }
	{ bl }
	{ }
	{ \mymathbold }
	\use:n
	
	\bamboo_define:nnnnnN
	{ \char_generate:nn { #1 } { 11 } }
	{ \char_generate:nn { #1 } { 11 } }
	{ scr }
	{ }
	{ \mathcal }
	\use:n
	
	\bamboo_define:nnnnnN
	{ \char_generate:nn { #1 } { 11 } }
	{ \char_generate:nn { #1 } { 11 } }
	{ }
	{ hat }
	{ \widehat }
	\use:n
	
	\bamboo_define:nnnnnN
	{ \char_generate:nn { #1 } { 11 } }
	{ \char_generate:nn { #1 } { 11 } }
	{ }
	{ br }
	{ \overline }
	\use:n
	
	\bamboo_define:nnnnnN
	{ \char_generate:nn { #1 } { 11 } }
	{ \char_generate:nn { #1 } { 11 } }
	{ }
	{ tl }
	{ \widetilde }
	\use:n
	
	\bamboo_define:nnnnnN
	{ \char_generate:nn { #1 } { 11 } }
	{ \char_generate:nn { #1 } { 11 } }
	{ scr }
	{ br }
	{ \scrbar }
	\use:n
}
\int_step_inline:nnn { `a } { `z }
{
	\bamboo_define:nnnnnN
	{ \char_generate:nn { #1 } { 11 } }
	{ \char_generate:nn { #1 } { 11 } }
	{ bl }
	{ }
	{ \mymathbold }
	\use:n
	
	\bamboo_define:nnnnnN
	{ \char_generate:nn { #1 } { 11 } }
	{ \char_generate:nn { #1 } { 11 } }
	{ }
	{ hat }
	{ \hat }
	\use:n
	
	\bamboo_define:nnnnnN
	{ \char_generate:nn { #1 } { 11 } }
	{ \char_generate:nn { #1 } { 11 } }
	{ }
	{ br }
	{ \bar }
	\use:n
	
	\bamboo_define:nnnnnN
	{ \char_generate:nn { #1 } { 11 } }
	{ \char_generate:nn { #1 } { 11 } }
	{ }
	{ tl }
	{ \tilde }
	\use:n                            
}
\clist_map_inline:nn
{
	Gamma,Delta,Theta,Lambda,Xi,Pi,Sigma,Phi,Psi,Omega
}
{
	\bamboo_define:nnnnnN
	{ #1 }
	{ #1 }
	{ bl }
	{ }
	{ \mymathbold }
	\use:c
	
	\bamboo_define:nnnnnN
	{ #1 }
	{ #1 }
	{ op }
	{ }
	{ \op }
	\use:c
	
	\bamboo_define:nnnnnN
	{ #1 }
	{ #1 }
	{ scr }
	{ }
	{ \mathcal }
	\use:c
	
	\bamboo_define:nnnnnN
	{ #1 }
	{ #1 }
	{ }
	{ hat }
	{ \widehat }
	\use:c
	
	\bamboo_define:nnnnnN
	{ #1 }
	{ #1 }
	{ }
	{ br }
	{ \overline }
	\use:c
	
	\bamboo_define:nnnnnN
	{ #1 }
	{ #1 }
	{ }
	{ tl }
	{ \widetilde }
	\use:c
}
\clist_map_inline:nn
{
	alpha,beta,gamma,delta,epsilon,zeta,eta,theta,iota,kappa,
	lambda,mu,nu,xi,pi,rho,sigma,tau,phi,chi,psi,omega
}
{
	\bamboo_define:nnnnnN
	{ #1 }
	{ #1 }
	{ bl }
	{ }
	{ \mymathbold }
	\use:c
	
	\bamboo_define:nnnnnN
	{ #1 }
	{ #1 }
	{ }
	{ hat }
	{ \hat }
	\use:c
	
	\bamboo_define:nnnnnN
	{ #1 }
	{ #1 }
	{ }
	{ br }
	{ \bar }
	\use:c
	
	\bamboo_define:nnnnnN
	{ #1 }
	{ #1 }
	{ }
	{ tl }
	{ \tilde }
	\use:c
}

\ExplSyntaxOff

\DeclareMathOperator*{\union}{\bigcup}

\DeclareMathOperator{\E}{\mathbf{E}}
\renewcommand{\P}{\operatorname{\mymathbold{P}}}

\newcommand{\tr}{\operatorname{tr}}
\newcommand{\argmin}{\operatornamewithlimits{arg~min}}
\newcommand{\argmax}{\operatornamewithlimits{arg~max}}

\DeclarePairedDelimiter{\norm}{\lVert}{\rVert}

\DeclarePairedDelimiter{\abs}{\lvert}{\rvert}

\DeclarePairedDelimiter{\braces}{\{}{\}}
\DeclarePairedDelimiter{\parens}{(}{)}
\DeclarePairedDelimiter{\brackets}{[}{]}
\DeclarePairedDelimiterX{\ip}[2]{\langle}{\rangle}{#1,#2}

\DeclarePairedDelimiterXPP{\normsub}[2]{}{\lVert}{\rVert}{_{#2}}{#1}
\DeclarePairedDelimiterXPP{\ipsub}[3]{}{\langle}{\rangle}{_{#3}}{#1,#2}

\DeclarePairedDelimiterXPP{\ipHS}[2]{}{\langle}{\rangle}{_{\mathrm{HS}}}{#1, #2}
\DeclarePairedDelimiterXPP{\normHS}[1]{}{\lVert}{\rVert}{_{\mathrm{HS}}}{#1}

\DeclarePairedDelimiterXPP{\ipF}[2]{}{\langle}{\rangle}{_{\mathrm{F}}}{#1, #2}
\DeclarePairedDelimiterXPP{\normF}[1]{}{\lVert}{\rVert}{_{\mathrm{F}}}{#1}

\DeclarePairedDelimiterXPP{\dkl}[2]{\operatorname{D_{KL}}}{(}{)}{}{#1 \: \delimsize\Vert \: #2}

\DeclarePairedDelimiterXPP{\restr}[2]{}{{}}{\vert}{_{#2}}{#1}

\newcommand{\R}{\mathbf{R}}

\newcommand{\var}{\operatorname{var}}

\newcommand{\spn}{\operatorname{span}}

\newcommand{\normaldist}{\operatorname{\mathcal{N}}}
\newcommand{\poissondist}{\operatorname{Poisson}}

\newcommand{\given}{\:\vert\:}
\newcommand{\sign}{\operatorname{sign}}

\newcommand{\conv}{\operatorname{conv}}

\usepackage{microtype}
\usepackage{graphicx}
\graphicspath{{./figures/}}
\usepackage{subcaption}
\usepackage{fancyhdr}
\usepackage{algorithm}
\usepackage{algpseudocode}
\usepackage[backend=biber,doi=false,url=false,style=ieee]{biblatex}

\usepackage[capitalize]{cleveref}

\hypersetup{
	bookmarks=true,
	colorlinks,
	linkcolor=blue,
	citecolor=blue,
	urlcolor=blue,
}

\newtheorem{lemma}{Lemma}
\newtheorem{theorem}{Theorem}
\newtheorem{corollary}{Corollary}

\author{Andrew D.\ McRae \and Justin Romberg \and Mark A.\ Davenport%
}
\title{Optimal convex lifted sparse phase retrieval and PCA with an atomic matrix norm regularizer%
	\footnotetext{A.\ McRae is with the Institute of Mathematics, EPFL, Lausanne, Switzerland (e-mail: \texttt{andrew.mcrae@epfl.ch}). J.\ Romberg and M.\ Davenport are with the School of Electrical and Computer Engineering, Georgia Institute of Technology, Atlanta, Georgia, United States (e-mail: \texttt{jrom@ece.gatech.edu}, \texttt{mdav@gatech.edu}). This work was supported, in part, by NSF grants CCF-1718771 and CCF-2107455.}%
}

\DeclarePairedDelimiterXPP{\ipSigma}[2]{}{\langle}{\rangle}{_{\Sigma}}{#1,#2}
\DeclarePairedDelimiterXPP{\normSigma}[1]{}{\lVert}{\rVert}{_{\Sigma}}{#1}

\DeclarePairedDelimiterXPP{\normMixeds}[1]{}{\lVert}{\rVert}{_{*,s}}{#1}
\DeclarePairedDelimiterXPP{\normpsitwo}[1]{}{\lVert}{\rVert}{_{\psi_2}}{#1}

\newcommand{\projbp}{\scrP_{\beta^\perp}}
\newcommand{\projb}{\scrP_{\beta}}
\newcommand{\projI}{\scrP_I}
\newcommand{\projIp}{\scrP_{I^\perp}}
\newcommand{\projTpIp}{\scrP_{T^\perp \cap I^\perp}}
\newcommand{\projTIp}{\scrP_{T \cap I^\perp}}
\newcommand{\projTpI}{\scrP_{T^\perp \cap I}}
\newcommand{\projT}{\scrP_{T}}
\newcommand{\projTp}{\scrP_{T^\perp}}
\newcommand{\projTI}{\scrP_{T \cap I}}
\newcommand{\projbpI}{\scrP_{\beta^\perp \cap I}}

\theoremstyle{definition}
\newtheorem{assumption}{Assumption}
\Crefname{assumption}{Assumption}{Assumptions}
\newtheorem{remark}{Remark}

\begin{document}
	\maketitle
	
	\begin{abstract}
		We present novel analysis and algorithms for solving sparse phase retrieval and sparse principal component analysis (PCA) with convex lifted matrix formulations. The key innovation is a new mixed atomic matrix norm that, when used as regularization, promotes low-rank matrices with sparse factors. We show that convex programs with this atomic norm as a regularizer provide near-optimal sample complexity and error rate guarantees for sparse phase retrieval and sparse PCA. While we do not know how to solve the convex programs exactly with an efficient algorithm, for the phase retrieval case we carefully analyze the program and its dual and thereby derive a practical heuristic algorithm. We show empirically that this practical algorithm performs similarly to existing state-of-the-art algorithms.
	\end{abstract}
	\section{Introduction}
	\label{sec:intro}
	\subsection{Sparsity, phase retrieval, and PCA}
	\label{sec:main_contrib}
	Consider the standard linear regression problem in which we make observations of the form $y_i = \ip{x_i}{\beta^*} + \xi_i$, $i = 1, \dots, n$,
	where $\beta^* \in \R^p$ is a vector we want to estimate,
	$x_1, \dots, x_n \in \R^p$ are measurement vectors, and $\xi_1, \dots, \xi_n$ represent noise or other error.
	If the $x_i$'s are chosen randomly and independently (e.g., i.i.d.\ Gaussian),
	and the noise is zero-mean and independent with $\var(\xi_i) \leq \sigma^2$,
	it is well-known that in general, we need\footnote{Here and throughout the paper, $\lesssim$ and $\gtrsim$ denote, respectively, $\leq$ and $\geq$ within absolute constants.}
	$n \gtrsim p$ measurements to estimate $\beta^*$ meaningfully,
	and the best possible error we can obtain is $\norm{\betahat - \beta^*}_2 \lesssim \sigma \sqrt{p/n}$.
		
	We can potentially do much better if we exploit \emph{sparsity} in the vector $\beta^*$.
	If $\beta^*$ has (at most) $s$ nonzero entries, the standard LASSO algorithm, which requires solving an $\ell_1$-regularized least-squares optimization problem,
	yields an estimator $\betahat$ satisfying $\norm{\betahat - \beta^*}_2 \lesssim \sigma \sqrt{(s/n) \log (p/s)}$ as long as the number of measurements satisfies $n \gtrsim s \log (p/s)$ (see, e.g., \cite[Chapter 10]{Vershynin2018}).
	Thus by using a convex regularized optimization problem we can exploit sparsity to reduce the number of measurements $n$ and the estimation error proportionally to sparsity level (i.e., the number of nonzero entries in $\beta^*$).
	In this paper, we seek to extend this phenomenon to two problems: \emph{phase retrieval} and \emph{principal component analysis} (PCA).
	To introduce our main results, we briefly describe phase retrieval and PCA and their sparse variants.
	We focus on the formulations most relevant to our results.
	More complete background and related literature can be found in \Cref{sec:pr_intro,sec:pca_intro}.
	
	In phase retrieval, we seek to estimate a vector $\beta^*$ from $n$ noisy \emph{quadratic} observations of the form $y_i = \abs{\ip{x_i}{\beta^*}}^2 + \xi_i$.
	The nonlinearity in the measurement model makes estimation and analysis more complicated than if our measurements are linear.
	To get around this, a common approach is to note that for any $x, \beta \in \R^p$, $\abs{\ip{x}{\beta}}^2 = \ipHS{X}{B}$,
	where $X = x \otimes x$ and $B = \beta \otimes \beta$ are rank-1 positive semidefinite (PSD) matrices,
	and $\ipHS{\cdot}{\cdot}$ denotes the Hilbert-Schmidt (Frobenius) matrix inner product.
	We can then write our observations as the \emph{linear} measurements $y_i = \ipHS{X_i}{B^*} + \xi_i$,
	where $B^* = \beta^* \otimes \beta^*$ and $X_i = x_i \otimes x_i$.
	This is often called a ``lifted'' formulation, since we are mapping the parameter of interest from $\R^p$ to the larger space of $p\times p$ PSD matrices.
	If the $x_i$'s are randomly chosen (say, Gaussian),
	and we solve the semidefinite program
	\[
		\Bhat = \argmin_{B \succeq 0}~\frac{1}{2n} \sum_{i=1}^n (y_i - \ipHS{X_i}{B})^2,
	\]
	we can bound $\normHS{\Bhat - B^*} \lesssim \sigma \sqrt{p/n}$
	as long as $n \gtrsim p$, where $\sigma$ is the standard deviation of the $\xi_i$'s.
	(As shown in \cite{Candes2012}, this implies that the leading eigenvector of $\Bhat$ is close to $\beta^*$ up to its sign.)
	Both the sample complexity and the error rate are comparable to those in ordinary linear regression.
	
	In PCA,
	we observe $n$ i.i.d.\ random vectors $\{x_i\}_{i=1}^n$,
	and we want to estimate the leading eigenvector $v_1$ of the covariance matrix $\Sigma = \E(x_1 \otimes x_1)$.
	Again, this can be solved in a lifted manner with a semidefinite program, noting that
	\[
		P_1 \coloneqq v_1 \otimes v_1
		= \argmax_{P \in \R^{p \times p}}~\ipHS{\Sigma}{P}~\text{s.t.}~\norm{P}_* \leq 1.
	\]
	An estimator $\Phat$ of $P_1$ is obtained\footnote{It would be computationally suboptimal in practice to compute the leading eigenvector of $\Sigmahat$ with a semidefinite program, but this formulation helps motivate our approach to the sparse case.} by replacing $\Sigma$ with the empirical covariance $\Sigmahat$.
	Again, if $n \gtrsim p$, we can recover $P_1$ within error proportional to $\sqrt{p/n}$ (where the constants depend on the gap between the first and second leading eigenvalues of $\Sigma$).
	
	\emph{Sparse phase retrieval} seeks to combine phase retrieval with sparse recovery.
	If $\beta^*$ is $s$-sparse,
	and we observe $y_i = \abs{\ip{x_i}{\beta^*}}^2 + \xi_i$ for $i \in \{1, \dots, n\}$,
	can we recover $\beta^*$ with a similar sample complexity and error as in linear sparse recovery?
	Similarly, the question we consider in \emph{sparse PCA} is whether, if the leading eigenvector $v_1$ is $s$-sparse,
	we can recover it with a similar sample complexity and error as in linear recovery.
	
	Our main contributions are the following:
	\begin{itemize}
		\item We present novel convex relaxations of the sparse phase retrieval and sparse PCA problems that use both a lifted formulation and a sparsity-inducing regularization,
		and we prove that for both problems, an estimator computed via a convex program achieves an $O(s \log (p/s))$ sample complexity as in linear sparse recovery.
		Furthermore, in both problems, the estimators achieve the optimal $O(\sqrt{(s/n) \log (p/s)})$ error rate (with the caveat, for the sparse phase retrieval problem with unbounded noise, that $n$ may need to be larger than the minimum sample complexity to obtain this optimal rate).
		\item Although we do not know how to compute the convex programs exactly (we suspect they may, in fact, be computationally intractable), we present a heuristic motivated by a careful analysis of the dual problem and the problem's optimality conditions,
		and we show that in the case of sparse phase retrieval,
		the resulting algorithm achieves nearly identical empirical performance to existing state-of-the-art sparse phase retrieval algorithms.
	\end{itemize}
	In the following sections, we describe the sparse phase retrieval and sparse PCA problems in more detail,
	and we review the related literature.

	\subsection{Sparse phase retrieval}
	\label{sec:pr_intro}
	Phase retrieval in $p$ dimensions with (sub-)Gaussian measurements is by now well-studied.
	If we have $n$ observations of the form $y_i \approx \abs{\ip{x_i}{\beta^*}}^2$,
	we can solve the optimization problem
	\begin{equation}
		\label{eq:pr_basic_ncvx}
		\betahat = \argmin_{\beta \in \R^p}~ \sum_{i=1}^n (y_i - \abs{\ip{x_i}{\beta^*}}^2)^2.
	\end{equation}
	Unfortunately, this is a nonconvex problem,
	so there is no immediately obvious way to solve it efficiently.
	(A similar optimization problem and similar nonconvexity appear if we instead write our measurements without the square, i.e., our observations are $\approx \abs{\ip{x_i}{\beta^*}}$.)
	
	Most approaches to this algorithmic difficulty fall into one of two categories.
	One method is to optimize a nonconvex loss function such as \eqref{eq:pr_basic_ncvx} directly (and iteratively) with a suitable initialization (e.g., \cite{Candes2015}).
	The other is the lifted semidefinite approach outlined in \Cref{sec:main_contrib}.
	For example, \textcite{Candes2013} show that if the design vectors $x_i$ are Gaussian,
	$y_i = \abs{\ip{x_i}{\beta^*}}^2 + \xi_i$,
	and we have $n \gtrsim p$ measurements,
	solving
	\[
		\Bhat = \argmin_{B \succeq 0}~\sum_{i=1}^n \abs{y_i - \ipHS{X_i}{B}}
	\]
	achieves $\norm{\Bhat - B^*}_F \lesssim \frac{1}{n} \sum_{i=1}^n \abs{\xi_i}$ with high probability.
	In the case of zero-mean random noise with standard deviation $\sigma$,
	we can, by using a squared loss, improve this to $\norm{\Bhat - B^*}_F \lesssim \sigma \sqrt{p/n}$ (see \cite{Thrampoulidis2019}).
	Thus we can solve the phase retrieval problem with a sample complexity and susceptibility to noise proportional to the dimension $p$;
	this is the same complexity as ordinary linear regression.
	
	Several results have been published on how to adapt iterative nonconvex phase retrieval algorithms to the sparse setting \cite{Netrapalli2013,Cai2016,Wang2018,Yuan2019,Yang2019,Jagatap2019}.
	Some \cite{Cai2016,Yang2019} do indeed achieve $O(\sigma \sqrt{(s/n) \log p})$ error bounds with zero-mean noise---this is very close to the optimal rate in linear sparse recovery (the rest do not analyze theoretically the noisy case).
	However, the theory in this literature requires $n \gtrsim s^2 \log p$,
	which, unless $s$ is very small, is much larger than what is required in linear sparse recovery.
	As \textcite{Soltanolkotabi2019} points out, the key difficulty is finding a good initialization for the algorithms---once we are close enough to $\beta^*$,
	we only need\footnote{Here and hereafter, $\gtrsim_{\log}$ ($\lesssim_{\log}$) will denote ``greater (less) than within a logarithmic factor.''} $n \gtrsim_{\log} s$ measurements to converge to a correct estimate.
	In practice, the first initialization step is often to estimate the support of $\beta^*$; the best known methods require $n \gtrsim_{\log} s^2$ measurements.
	We compare several of these algorithms (in addition to that of the purely algorithmic/empirical work \cite{Schniter2015}) to ours empirically in \Cref{sec:sims},
	and we see that all of them appear empirically to have \emph{linear} sample complexity in $s$.
	Another similar iterative algorithm is given in \cite{Bakhshizadeh2020}; it has similar sample complexity requirements but, interestingly,
	it is derived from a more abstract compression-based algorithm that, though not practically computable, does obtain optimal $O(s)$ sample complexity.
	
	We see qualitatively similar sample complexity requirements in the works \cite{Hand2016,Salehi2018}, which extend to the sparse case the convex PhaseMax framework \cite{Bahmani2017,Goldstein2018}.
	Both results only require $n \gtrsim_{\log} s$ measurements if we already have an ``anchor'' vector $\beta_0 \in \R^p$ that has significant correlation with $\beta^*$.
	However, it is not known how to find such a $\beta_0$ (in a computationally efficient manner) without $n \gtrsim_{\log} s^2$ measurements.
	
	More related to our results are methods to adapt the lifted convex phase retrieval approach to the sparse setting.
	The foundational theoretical work in this area is by \textcite{Li2013}, although some work (mostly empirical) appeared in \cite{Ohlsson2012,Ohlsson2012a}.
	The key idea is that if $\beta^* \in \R^p$ is $s$-sparse,
	the lifted version $B^* = \beta^* \otimes \beta^*$ is both rank-1 and at most $s^2$-sparse.
	In the noiseless case, they solve the optimization problem
	\begin{equation}
		\label{eq:separate_reg}
		\Bhat = \argmin_{B \succeq 0}~\lambda_1 \tr(B) + \lambda_2 \norm{B}_{1,1} \text{ s.t. } \ipHS{X_i}{B} = y_i,\ i = 1,\dots, n,
	\end{equation}
	where $\norm{\cdot}_{1,1}$ denotes the elementwise $\ell_1$ norm of a matrix.
	The trace regularization term promotes low rank, while the $\ell_1$ norm promotes sparsity.
	As with the nonconvex methods, their theory requires $n \gtrsim s^2 \log p$ measurements to get exact recovery.
	The result of \cite{Thrampoulidis2019}, when specialized to sparse phase retrieval, extends this approach to the noisy case, getting, within log factors, the same $O(s^2)$ sample and noise complexity.
	
	Finally, we note that although we are primarily concerned with generic measurement vectors $x_i$ (e.g., sub-Gaussian),
	one can obtain better theoretically guaranteed sample complexity with practical algorithms if we have complete control over how the measurements are chosen; see, for example, \cite{Bahmani2015,Iwen2017}.
	\subsection{Sparse PCA}
	\label{sec:pca_intro}
	PCA is a well-established technique with which, given points $x_1, \dots, x_n \in \R^p$,
	we try to find a low-dimensional linear (or affine) subspace that contains most of the energy in the data.
	If $x_1, \dots, x_n$ have zero empirical mean (e.g., after centering),
	the closest $r$-dimensional subspace to the points (in mean square $\ell_2$ distance) is
	the space spanned by the top $r$ eigenvectors of the empirical covariance matrix $\Sigmahat = \frac{1}{n} \sum_{i=1}^n x_i \otimes x_i$.
	
	For simplicity, take $r = 1$.	
	Suppose the $x_i$'s are i.i.d.\ copies of a random variable $x$ with true covariance $\Sigma$ with eigenvalue decomposition $\Sigma = \sum_\ell \sigma_\ell v_\ell \otimes v_\ell$,
	where $\sigma_1 > \sigma_2 \geq \cdots \geq \sigma_p$.
	If $x$ is Gaussian, and $\sigma_2 \gtrsim \frac{\sigma_1}{p-1}$,
	then, with high probability \cite{Koltchinskii2017},
	\[
		\norm{\Sigmahat - \Sigma}_2 \lesssim\sqrt{\sigma_1 \frac{\sigma_1 + (p-1) \sigma_2}{n}} \lesssim \sqrt{\sigma_1 \sigma_2 \frac{p}{n}}.
	\]
	Then, if $\vhat_1$ is the leading eigenvector of $\Sigmahat$,
	the Davis-Kahan $\sin \Theta$ theorem gives 
	\[
		\norm{\vhat_1 \otimes \vhat_1 - v_1 \otimes v_1}_2 \lesssim \frac{\sqrt{\sigma_1 \sigma_2}}{\sigma_1 - \sigma_2} \sqrt{\frac{p}{n}}.
	\]
	This rate is minimax-optimal over general covariance matrices with the given $\sigma_1,\sigma_2$ (see \cite{Vu2012}).
	
	When $p$ is large compared to $n$, we need to impose more structure on $\Sigma$ to recover the leading eigenvector(s) accurately.
	In sparse PCA, we consider the case in which the eigenvector(s) of interest are \emph{sparse}.	
	This problem has been extensively studied in the past decade:
	see \cite{Zou2018} for a recent review.
	
	In the single-eigenvector recovery case ($r = 1$),
	\textcite{Cai2013} show that if the leading eigenvector $v_1$ is $s$-sparse,
	the minimax rate for all estimators $\vhat_1$ of $v_1$ over the simple class $\{\Sigma = \sigma_2 I_p + (\sigma_1 - \sigma_2) v_1 \otimes v_1 \colon v_1\ \text{$s$-sparse}, \norm{v_1}_2 = 1 \}$ is
	\[
		\norm{\vhat_1 \otimes \vhat_1 - v_1 \otimes v_1}_2 \approx \frac{\sqrt{\sigma_1 \sigma_2}}{\sigma_1 - \sigma_2} \sqrt{\frac{s \log (p/s)}{n}}.
	\]
	
	While this theoretical result is clean and achieves our desire to bring sparse-recovery sample complexity and error to the PCA problem,
	one practical problem remains:
	how do we \emph{compute} an estimator $\vhat_1$ that achieves these theoretical properties?
	The optimal estimator proposed in \cite{Cai2013} is, to quote that paper ``computationally intensive.''
	As with sparse phase retrieval, the best theoretical results for computationally efficient algorithms require $n \gtrsim_{\log} s^2$ to guarantee accurate recovery (see, e.g., \cite{Cai2013,Birnbaum2013}).
	Once again, proper initialization (often by estimating the support of $v_1$) is the key difficulty.
	
	There is strong evidence to suggest that this $s^2$ barrier may be inescapable for computationally efficient algorithms.
	Recent results suggest that any statistically optimal estimator that requires fewer measurements must be NP-hard to compute.
	\textcite{Berthet2013} showed that if a certain testing problem in random graph theory (the \emph{planted clique problem}) is NP-hard to compute in certain regimes (which is widely believed although so-far unproved in standard computational models), then accurately \emph{testing for the existence of} a sparse leading eigenvector when $n \lesssim_{\log} s^2$ is NP-hard.
	\textcite{Wang2016,Gao2017} further refine this by showing that, under a similar assumption, there is no efficiently computable consistent estimator of $v_1$ when $n \lesssim_{\log} s^2$.
	
	\section{Key tool: A sparsity-and-low-rank--inducing atomic norm}
	\label{sec:atomic}
	To motivate our approach,
	consider the optimization problem \eqref{eq:separate_reg} from \cite{Li2013} for sparse phase retrieval
	or its least-squares version
	\begin{equation}
		\label{eq:ls_norm_comb}
		\Bhat = \argmin_{B \succeq 0}~\frac{1}{2n} \sum_{i=1}^n (y_i - \ipHS{X_i}{B})^2 + \lambda_1 \tr(B) + \lambda_2 \norm{B}_{1,1}.
	\end{equation}
	It turns out that quadratic (in sparsity) $O(s^2)$ complexity is a fundamental performance bound for this class of methods.
	Our target matrix $B^*$ has two kinds of structure: it is rank-1 and $s^2$-sparse.
	The trace regularization in our estimator encourages low rank, while the $\ell_1$ regularization encourages sparsity.
	However, recent work \cite{Oymak2015,Kliesch2019} has shown it is impossible to take advantage of both kinds of structure simultaneously with a regularizer that is merely a convex combination of the two structure-inducing regularizers;
	the best we can do is exploit either the low rank as in non-sparse phase retrieval,
	in which case we get $O(p)$ complexity, or the $s^2$-sparsity, in which case we get $O(s^2)$ complexity.
	
	To see intuitively why we have this problem, note that the nuclear norm and elementwise $\ell_1$ norm are both examples of \emph{projective tensor norms} \cite{Diestel2002}.
	For matrix $A$ of any size,
	\[
		\norm{A}_* = \inf~ \braces*{ \sum \norm{u_k}_2 \norm{v_k}_2 : A = \sum u_k \otimes v_k}
	\]
	and
	\[
		\norm{A}_{1,1} = \inf~ \braces*{ \sum \norm{u_k}_1 \norm{v_k}_1 : A = \sum u_k \otimes v_k}
	\]
	Equivalently, these norms are atomic norms \cite{Chandrasekaran2012} where the atoms are rank-1 matrices with unit $\ell_2$ or $\ell_1$ norms.
	For a PSD matrix, the trace is the nuclear norm,
	so the regularizer in \eqref{eq:ls_norm_comb} can be expressed as
	\begin{equation}
		\label{eq:cvx_reg}
		\begin{aligned}
			\lambda_1 \tr(B) + \lambda_2 \norm{B}_{1,1}
			&= \lambda_1 \inf~ \braces*{ \sum \norm{u_k}_2 \norm{v_k}_2 : B = \sum u_k \otimes v_k} \\
			&\qquad+ \lambda_2 \inf~ \braces*{ \sum \norm{w_k}_1 \norm{z_k}_1 : B = \sum w_k \otimes z_k}.
		\end{aligned}
	\end{equation}
	A key feature of $B^* = \beta^* \otimes \beta^*$ is that the factors of its rank-1 decomposition have a certain $\ell_2$ norm \emph{and} are sparse.
	Because the two infima in \eqref{eq:cvx_reg} are separate, the regularizer promotes matrices with two \emph{separate} atomic decompositions of low $\ell_2$ and $\ell_1$ norm respectively.
	It does not encourage a decomposition into low-rank matrices with factors that have \emph{simultaneously} low $\ell_2$ norm and low $\ell_1$ norm.
	
	Inspired by the framework of \textcite{Haeffele2020},
	we propose the following regularizer:
	\begin{equation}
		\label{eq:mixednorm_def}
		\normMixeds{B} \coloneqq \inf~\braces*{ \sum \theta_s(u_k, v_k) : B = \sum u_k \otimes v_k },
	\end{equation}
	where
	\[
		\theta_s(u, v) = \parens*{ \norm{u}_2 + \frac{1}{\sqrt{s}} \norm{u}_1 }\parens*{ \norm{v}_2 + \frac{1}{\sqrt{s}} \norm{v}_1},
	\]
	and $s > 0$ is a parameter that represents the sparsity (or an approximation thereof) of the vector we are interested in recovering.
	For some intuition on this choice of regularizer, note that
	\[
		\{ A: \normMixeds{A} \leq 1 \} \approx  \conv\{ u \otimes v : \norm{u}_2 = \norm{v}_2 = 1, u, v \text{ are $s$-sparse} \},
	\]
	by which we mean that either is contained within a modest scaled version of the other.
	One direction is a simple consequence of the fact that for an $s$-sparse vector $u$, $\norm{u}_1 \leq \sqrt{s} \norm{u}_2$.
	The other direction is provided by \Cref{lem:norm_atomic_equiv} in \Cref{app:mixednorm}.
	Thus $\normMixeds{\cdot}$ is (equivalent to) an atomic norm whose atoms are precisely the type of matrix we expect $B^*$ to be.\footnote{If we ``guess wrongly'' the sparsity of $\beta^*$, we can still get similar results with different constants of equivalence.}
	Similar notions of atomic norms that promote simultaneous low rank and sparsity have appeared in \cite{Richard2014,Kliesch2019}.
	
	We will show in the next section that using $\normMixeds{\cdot}$ as a regularizer in lifted formulations of sparse phase retrieval and PCA
	gives sample complexity and error bounds nearly identical to the linear regression case.
	
	\section{Theoretical guarantees for atomic-norm regularized estimators}
	\label{sec:theorems}
	In this section, we state precisely our main problems, assumptions, abstract convex optimization algorithm, and theoretical guarantees.
	
	\subsection{Sparse phase retrieval}
	\label{sec:pr_theory}
	
	Suppose $\beta^* \in \R^p$ is an $s$-sparse vector.
	Let $x$ be a random vector in $\R^p$.
	We observe $n$ i.i.d.\ copies $(x_1, y_1), \dots, (x_n, y_n)$ of the random couple $(x, y)$, where $y$ is a real random variable whose distribution conditioned on $x$ depends only on $\ip{x}{\beta^*}^2$ (i.e., $y \sim p_y(y \given \ip{x}{\beta^*}^2)$).
	Let $\xi \coloneqq y - \ip{x}{\beta^*}^2$ denote the ``noise.'' We make the following assumptions:
	
	\begin{assumption}[Sub-Gaussian measurements]
		\label{assump:spr_meas}
		The entries $(x^{(1)}, \dots, x^{(p)})$ of $x$ are i.i.d.\ real random variables with $\E x^{(\ell)} = 0$, $\E (x^{(\ell)})^2 = 1$, $\E (x^{(\ell)})^4 > 1$, and sub-Gaussian norm $\norm{x^{(\ell)}}_{\psi_2} \leq K$ for some $K > 0$.
	\end{assumption}
	Note that the fourth-moment assumption excludes Rademacher random variables.
	In what follows, for simplicity of presentation, all dependence on $K$ and the difference $\E (x^{(\ell)})^4 - 1$ will be subsumed into unspecified constants.
	
	\begin{assumption}[Zero-mean, bounded-moment noise]
		\label{assump:spr_noise}
		$\E [\xi \given x] = 0$ almost surely, and, for all $u \in \R^p$ such that $\norm{u}_2 \leq 1$,
		\[
			\E \xi^2 \ip{x}{u}^4 \leq \sigma^2(\beta^*),
		\]
		where $\sigma^2(\beta^*)$ is a quantity that possibly depends on the vector $\beta^*$, the distribution of $x$, and the conditional distribution of $y$.
		Furthermore, there are $M, \eta \geq 0$ such that
		\[
			\norm{\xi \ip{x}{u}^2}_\alpha \leq M \alpha^{\eta + 1}
		\]
		for $\alpha \geq 3$ and all $u \in \R^p$ such that $\norm{u}_2 \leq 1$
		(where $\norm{Z}_\alpha \coloneqq (\E \abs{Z}^\alpha)^{1/\alpha}$ for any random variable $Z$).
	\end{assumption}
	Our two working examples are the following:
	\begin{itemize}
		\item Independent additive noise: $\xi$ is independent of all other quantities, in which case we can take $\sigma^2(\beta) \approx \var(\xi)$,
		and $M$ and $\eta$ depend on the moments of $\xi$.
		
		\item Poisson noise: $y \sim \poissondist(\ip{x}{\beta^*}^2)$ conditioned on $x$.
		In this case, under \Cref{assump:spr_meas}, we can take $\sigma^2(\beta^*) \approx \norm{\beta^*}_2^2$,
		$M \approx \norm{\beta^*}_2 + 1$, and $\eta = 1$ (we prove this in \Cref{app:poisson}).
	\end{itemize}
	
	As before, we lift the problem into the space of PSD matrices by setting $B^* = \beta^* \otimes \beta^*$ and $X = x \otimes x$.
	We then choose a regularization parameter $\lambda \geq 0$ and compute our estimate by the following optimization problem:
	\begin{equation}
		\label{eq:pr_opt_main}
		\Bhat = \argmin_{B \in \R^{p\times p}}~\frac{1}{2n} \sum_{i=1}^n (y_i - \ipHS{X_i}{B})^2 + \lambda \normMixeds{B}.
	\end{equation}
	
	We then have the following guarantee for sample complexity and error, proved in \Cref{sec:pr_proof}:
	\begin{theorem}
		\label{thm:pr}
		Suppose \Cref{assump:spr_meas,assump:spr_noise} hold.
		Suppose $\beta^*$ is $s$-sparse and that the number of measurements $n$ satisfies $n \gtrsim s \log (ep/s)$.
		If the regularization parameter satisfies
		\[
			\lambda \gtrsim \sqrt{\frac{s \log(ep/s)}{n} \sigma^2(\beta^*)} + \frac{M}{n^{1-c}} \parens*{s \log \frac{ep}{s}}^{\eta+1},
		\]
		where $c \approx (s \log (ep/s))^{-1}$,
		then, with probability at least $1 - e^{-bn} - e^{-s} (s/p)^s$ (where $b > 0$ is a constant),
		the estimator $\Bhat$ from \eqref{eq:pr_opt_main} satisfies
		\[
			\norm{\Bhat - B^*}_* \lesssim \lambda.
		\]
	\end{theorem}
	\begin{remark}
		For simplicity of presentation, we assume that the sparsity level $s$ used in the regularizer is in fact (an upper bound on) the sparsity of $\beta^*$.
		We could easily extend our results to the ``misspecified'' case $\norm{\beta^*}_0 = s_0 > s$.
	\end{remark}

	\begin{remark}
		By a standard argument (found, e.g., in \cite{Candes2012}), if $\betahat \otimes \betahat$ is the closest rank-1 approximation to $\Bhat$, then $\betahat$ satisfies
		\[
		\min\{ \norm{\betahat - \beta^*}_2, \norm{\betahat + \beta^*}_2 \} \lesssim \frac{\lambda}{\norm{\beta^*}_2}.
		\]
	\end{remark}
	
	\begin{remark}
		The required sample complexity $s \log (ep/s)$ is precisely the optimal sample complexity from traditional linear sparse recovery.
		For large $n$, the noise error rate (with appropriately chosen $\lambda$) is also the optimal $\sqrt{(s/n) \log (ep/s)}$,
		but, if $\eta > 0$, achieving this rate may require $n$ to be significantly larger than $s \log (ep/s)$.
		More precisely, the first term containing the optimal rate will dominate if and only if
		\[
			n^{1-2c} \gtrsim \frac{M^2}{\sigma^2(\beta^*)} \parens*{ s \log \frac{ep}{s} }^{1 + 2\eta}.
		\]
		If the noise $\xi$ is bounded, we can take $\eta = 0$, and we only need $n^{1-2c} \gtrsim s \log \frac{ep}{s}$ to obtain the optimal error rate.
		For most interesting cases (where $c$ is very small), this is negligibly different from the sample complexity requirement.
		If $\xi$ is (conditionally) sub-Gaussian, we can take $\eta = 1/2$, in which case we need $n^{1-2c} \gtrsim_{\log} s^2$.
		If $\xi$ is (conditionally) sub-exponential, as in the Poisson noise case, we need $n^{1-2c} \gtrsim_{\log} s^3$.
		The need for larger $n$ comes (in our proof) from concentration inequalities for sums of terms of the form $\xi \ip{x}{u}^2$ for arbitrary vectors $u$; these terms have larger moments than the $\xi \ip{x}{u}$ terms we would typically see in linear settings.
		This could perhaps be improved with judicious truncation as in, for example, \cite{Chen2015b}.
	\end{remark}
	
	\begin{remark}
		In the independent additive noise case, one can check that our proof gives a high-probability bound uniform over $s$-sparse $\beta^*$.
		If $\var(\xi) = \sigma^2$, we get, for appropriately chosen $\lambda$,
		\[
		\norm{\Bhat - B^*}_* \lesssim \sqrt{\frac{s \log(ep/s)}{n}} \sigma + \frac{M}{n^{1-c}} \parens*{s \log \frac{ep}{s}}^{\eta+1}.
		\]
	\end{remark}
	
	\begin{remark}
		In the Poisson observation case, we obtain, for appropriately chosen $\lambda$,
		\[
		\norm{\Bhat - B^*}_* \lesssim \sqrt{\frac{s \log (ep/s)}{n}} \norm{\beta^*}_2 + \frac{\norm{\beta^*}_2 + 1}{n^{1-c}} \parens*{ s \log \frac{ep}{s}}^2.
		\]
		When $\beta^* \neq 0$, and $n$ is large enough that the first error term dominates,
		we have, up to a sign, that
		\[
			\norm{\betahat - \beta^*}_2 \lesssim \sqrt{\frac{s \log (ep/s)}{n}},
		\]
		where $\betahat$ is the appropriately-scaled leading eigenvector of $\Bhat$.
		Thus we get an error bound does that not depend on $\norm{\beta^*}_2$.
	\end{remark}

	\begin{remark}
		If there is no noise ($\xi = 0$),
		our analysis could easily be adapted to study the problem
		\[
			\min_B~\normMixeds{B} \text{ s.t. } \ipHS{X_i}{B} = y_i, i = 1,\dots,n.
		\]
		To understand how to use our proof techniques, note that any solution $\Bhat$ to the above problem satisfies $\sum_{i=1}^n \ipHS{X_i}{H}^2 = 0$
		and
		\[
			0 \geq \normMixeds{\Bhat} - \normMixeds{B^*} \geq \ipHS{W_{B^*}}{H},
		\]
		for any subgradient $W_{B^*} \in \partial \normMixeds{B^*}$,
		where $H = \Bhat - B$.
	\end{remark}
	
	\subsection{Sparse PCA}
	\label{sec:pca_theory}
		We can apply the atomic regularizer to the sparse PCA problem via another standard lifted formulation:
	\begin{theorem}
		\label{thm:pca}
		Suppose we observe $n$ i.i.d.\ copies of the $p$-dimensional vector $x \sim \normaldist(\mu, \Sigma)$,
		where $\Sigma = \sigma_1 v_1 \otimes v_1 + \Sigma_2$,
		$v_1$ is $s$-sparse and unit-norm, $\sigma_1 > \norm{\Sigma_2} \eqqcolon \sigma_2$, and $\Sigma_2 v_1 = 0$.
		Choose
		\[
			\lambda \gtrsim \sqrt{\sigma_1 \sigma_2} \sqrt{\frac{s \log (ep/s)}{n}}
		\]
		and let
		\begin{equation}
			\label{eq:pca_opt}
			\Phat = \argmin_{P \in \R^{p \times p}}~ - \ipHS{\Sigmahat}{P} + \lambda \normMixeds{P} \ \mathrm{s.t.} \ \norm{P}_* \leq 1,
		\end{equation}
		where
		\[
			\Sigmahat = \frac{1}{n} \sum_{i=1}^n (x_i - \xbr) \otimes (x_i - \xbr) = \parens*{\frac{1}{n} \sum_{i=1}^n x_i \otimes x_i} - \xbr \otimes \xbr
		\]
		is the empirical covariance of $x_1, \dots, x_n$ ($\xbr = \frac{1}{n} \sum_{i=1}^n x_i$).
		
		For $t > 0$, if $n \gtrsim \max\braces*{ s \log \frac{ep}{s}, \parens*{\frac{\sigma_1}{\sigma_1 - \sigma_2}}^2 t }$,
		then, with probability at least $1 - e^{-t} - 3 e^{-s}(s/p)^s$,
		\[
			\norm{\Phat - P_1}_F \lesssim \frac{\lambda}{\sigma_1 - \sigma_2},
		\]
		where $P_1 = v_1 \otimes v_1$.
	\end{theorem}
	We prove this result fully in \Cref{app:pca_proof}. A sketch of the proof is provided in \Cref{sec:pca_proof_sketch}.
	
	\begin{remark}
		The assumption that $x$ is Gaussian could easily be relaxed to $x = \Sigma^{1/2} z$,
		where $z$ is a sub-Gaussian random vector, as in, for example, \cite{Vu2012}.
	\end{remark}
	\begin{remark}
		For properly chosen $\lambda$ the resulting error rate
		\[
			\norm{\Phat - P_1}_F \lesssim \frac{\sqrt{\sigma_1 \sigma_2}}{\sigma_1 - \sigma_2} \sqrt{\frac{s \log (ep/s)}{n}}
		\]
		matches the minimax lower bounds in \cite{Vu2012,Cai2013}.
	\end{remark}

	\subsection{PSD constraints and another regularizer}
	For phase retrieval and PCA,
	it is natural to restrict our estimators to be PSD.
	All of our theoretical results hold if we add a $B \succeq 0$ constraint to \eqref{eq:pr_opt_main} or a $P \succeq 0$ constraint to \eqref{eq:pca_opt}.
	
	Unlike the nuclear norm case (where the optimal decomposition is the singular value decoposition, which is identical to the eigenvalue decomposition for a PSD matrix),
	it is not clear whether every PSD matrix $B$ admits a symmetric (i.e., $u_k = v_k$) optimal decomposition with regard the definition of $\normMixeds{B}$ in \eqref{eq:mixednorm_def}.
	Therefore, it is natural to define as a new regularizer the following gauge function/asymmetric norm on the space of PSD matrices:
	for $B \succeq 0$,
	\[
		\Theta_s(B) = \inf~\braces*{ \sum \theta_s(u_k, u_k) : B = \sum u_k \otimes u_k }.
	\]
	All of our theoretical and computational results in \Cref{sec:theorems,sec:opt} can be easily extended to this choice of regularizer.
	This choice of regularizer is computationally convenient because if we optimize over a matrix $B$ by optimizing over factors $u_k, v_k$ such that $B = \sum_k u_k \otimes v_k$ (see \Cref{sec:practical_alg}),
	we can enforce a PSD constraint simply by forcing $u_k = v_k$.
	
	\section{Proof highlights}
	In this section, we outline the proofs of \Cref{thm:pr,thm:pca}.
	We fully prove \Cref{thm:pr} from some technical lemmas, while we sketch the proof of \Cref{thm:pca}
	
	\subsection{Sparse phase retrieval proof}
	\label{sec:pr_proof}
	In this section, we prove \Cref{thm:pr}, which is our error bound for sparse phase retrieval.
	We will use the following key technical lemmas:
	\begin{lemma}[Subgradients of mixed atomic norm]
		\label{lem:subgrad_ineq}
		Suppose $\beta \in \R^p$ is $s$-sparse, and let $B = \beta \otimes \beta$.
		Then, for every matrix $A \in \R^{p \times p}$,
		there exists $W \in \partial \normMixeds{B}$ such that
		\[
		\ipHS{W}{A} \geq \frac{1}{10} \normMixeds{A} - 5 \norm{A}_F.
		\]
	\end{lemma}
	
	\begin{lemma}[Empirical process bound]
		\label{lem:emp_proc}
		Let $G_1, \dots, G_n$ be i.i.d.\ copies of a random matrix $G \in \R^{p \times p}$,
		where, for all $u, v \in \R^p$, $\ip{Gu}{v}$ has zero mean,
		\[
		\E \ip{G u}{v}^2 \leq \sigma^2 \norm{u}_2^2 \norm{v}_2^2,
		\]
		and
		\[
		\norm{\ip{G u}{v}}_\alpha \leq M \alpha^{\eta+1} \norm{u}_2 \norm{v}_2
		\]
		for all $\alpha \geq 3$.
		
		Let	$Z = \frac{1}{n} \sum_{i=1}^n G_i$.
		For $s \geq 1$, with probability at least $1-e^{-s}(s/p)^s$,
		\[
		\sup_{\normMixeds{A} \leq 1}~\ipHS{Z}{A}
		\lesssim \sigma \sqrt{\frac{s \log (e p/s)}{n}} + \frac{M}{n^{1-c}} \parens*{s \log \frac{ep}{s}}^{\eta + 1},
		\]
		where $c \approx \frac{1}{s \log (ep/s)}$.
	\end{lemma}

	\begin{lemma}[Restricted lower isometry]
		\label{lem:rsc}
		Let $x_1, \dots, x_n$ be i.i.d.\ copies of a random vector $x$ satisfying \Cref{assump:spr_meas},
		and let $X_i = x_i \otimes x_i$.
		Suppose
		\[
		n \gtrsim s \log \frac{e p}{s},
		\]
		and let $C \geq 1$ be a fixed constant.
		With probability at least $1 - e^{-bn}$ (for some $b > 0$), the following event holds:
		For all $A \in \R^{p \times p}$ such that
		\[
		\normMixeds{A} \leq C \norm{A}_F,
		\]
		we have
		\[
		\frac{1}{n} \sum_{i=1}^n \ipHS{X_i}{A}^2 \gtrsim \norm{A}_F^2,
		\]
		where the constant in the lower bound depends on $C$.
	\end{lemma}
	\Cref{lem:subgrad_ineq} is proved in \Cref{app:mixednorm}. \Cref{lem:emp_proc,lem:rsc} are proved in \Cref{app:empirical_proofs}.
	With these, we can prove the sparse phase retrieval error bound:	
	\begin{proof}[Proof of \Cref{thm:pr}]
		Applying \Cref{lem:emp_proc} to the random matrices $G_i = \xi_i X_i$, we can choose $\lambda$ according to the theorem statement with large enough constant
		so that, with probability at least $1 - e^{-s} (s/p)^s$,
		\[
		\sup_{\normMixeds{A} \leq 1} \ipHS*{\frac{1}{n} \sum_{i=1}^n \xi_i X_i }{A} \leq \frac{\lambda}{20}.
		\]
		Furthermore, by \Cref{lem:rsc}, for $n \gtrsim s \log \frac{ep}{s}$ (with large enough constant),
		we have, with probability at least $1 - e^{-bn}$,
		\[
		\frac{1}{n} \sum_{i=1}^n \ipHS{X_i}{A}^2 \gtrsim \norm{A}_F^2
		\]
		for all $A$ satisfying $\normMixeds{A} \leq 100 \norm{A}_F$.
		
		The intersection of these events occurs with probability at least $1 - e^{-s} (s/p)^s - e^{-bn}$.
		In what follows, we assume this holds.
		
		Let $\Bhat$ be the solution to \eqref{eq:pr_opt_main}.
		Writing $F(B)$ as the objective function,
		the convexity of the optimization problem implies that
		\[
		0 \leq \ipHS{\nabla F(\Bhat)}{B^* - \Bhat} = \frac{1}{n} \sum_{i=1}^n (y_i - \ipHS{X_i}{\Bhat})\ipHS{X_i}{\Bhat - B^*} + \lambda \ipHS{W_{\Bhat}}{B^* - \Bhat},
		\]
		for any $W_{\Bhat} \in \partial \normMixeds{\Bhat}$.
		By the monotonicity of (sub)gradients of convex functions,
		we have that, for any $W \in \partial \normMixeds{B^*}$,
		$\ipHS{W - W_{\Bhat}}{B^* - \Bhat} \geq 0$, and therefore
		\[
		0 \leq \frac{1}{n} \sum_{i=1}^n (y_i - \ipHS{X_i}{\Bhat}) \ipHS{X_i}{\Bhat - B^*} + \lambda \ipHS{W}{B^* - \Bhat}.
		\]
		Let $H = \Bhat - B^*$.
		Using the fact that $(y_i - \ipHS{X_i}{\Bhat}) \ipHS{X_i}{\Bhat - B^*} = \xi_i \ipHS{X_i}{H} - \ipHS{X_i}{H}^2$, we have
		\[
		\frac{1}{n} \sum_{i=1}^n \ipHS{X_i}{H}^2 \leq \frac{1}{n} \sum_{i=1}^n \xi_i \ipHS{X}{H} - \lambda \ipHS{W}{H}
		\leq \frac{\lambda}{20} \normMixeds{H} - \lambda \ipHS{W}{H}.
		\]
		By \Cref{lem:subgrad_ineq}, there exists $W \in \partial \normMixeds{B^*}$ such that
		\[
		\ipHS{W}{H} \geq \frac{1}{10} \normMixeds{H} - 5 \norm{H}_F.
		\]
		Therefore, we have
		\[
		\frac{1}{n} \sum_{i=1}^n \ipHS{X_i}{H}^2 \leq 5 \lambda \norm{H}_F - \frac{\lambda}{20} \normMixeds{H}.
		\]
		Because the left side of this inequality is nonnegative,
		we have $\normMixeds{H} \leq 100 \norm{H}_F$.
		Then, by restricted lower isometry, we have
		\[
		\norm{H}_F^2 \lesssim \lambda \norm{H}_F.
		\]
		The result immediately follows.
		
	\end{proof}

	\subsection{Sparse PCA proof sketch}
	\label{sec:pca_proof_sketch}
	The proof of \Cref{thm:pca} is somewhat messier than the proof of \Cref{thm:pr} above,
	so we do not go into all of the details here. We refer the reader to \Cref{app:pca_proof} for the full proof.
	
	If $\Phat$ is an optimal solution of \eqref{eq:pca_opt},
	one can obtain, similarly to the proof of \Cref{thm:pr},
	that
	\[
		\ipHS{\Sigmahat}{H} \geq \lambda \ipHS{W}{H}
	\]
	for any $W \in \partial \normMixeds{P_1}$, where $H = \Phat - P_1$.
	Choosing $W$ according to \Cref{lem:subgrad_ineq},
	we obtain
	\[
		\ipHS{\Sigmahat}{H} \geq \lambda \parens*{ \frac{1}{10} \normMixeds{H} - 5 \norm{H}_F }.
	\]
	By analysis similar to \Cref{lem:emp_proc},
	one can show that
	\[
		\abs{\ipHS{\Sigmahat - \Sigma}{H}}
		\lesssim \sqrt{\sigma_1 \sigma_2 \frac{s \log (ep/s)}{n}} \normMixeds{H} + \sigma_1 \sqrt{\frac{t}{n}} \abs{\ipHS{H}{P_1}}
	\]
	with probability at least $1 - e^{-t} - 3 e^{-s} (s/p)^s$
	when $n \gtrsim s \log (ep/s)$.
	For $\lambda$ chosen so that the coefficient of $\normMixeds{H}$ above is $\leq \lambda/10$, we get, on this event,
	\[
		\ipHS{\Sigma}{H} \gtrsim -\lambda \norm{H}_F -\sigma_1 \sqrt{\frac{t}{n}} \abs{\ipHS{H}{P_1}}.
	\]
	Now, note that because $\norm{\Phat}_* \leq 1$, we have the following:
	\begin{itemize}
		\item $\abs{\ipHS{H}{P_1}} = 1 - \ipHS{\Phat}{P_1}$, and
		\item $\ipHS{\Sigma}{H} = \sigma_1 (\ipHS{\Phat}{P_1} - 1) + \ipHS{\Sigma_2}{\Phat} \leq \sigma_1 (\ipHS{\Phat}{P_1} - 1) + \sigma_2 (1 - \ipHS{\Phat}{P_1})$.
	\end{itemize}
	Then, using the assumption that $n \gtrsim \frac{\sigma_1^2}{(\sigma_1 - \sigma_2)^2} t$, we get
	\[
		(\sigma_1 - \sigma_2)(1 - \ipHS{\Phat}{P_1})
		\lesssim \parens*{\sigma_1 - \sigma_2 - \sigma_1 \sqrt{\frac{t}{n}}} (1 - \ipHS{\Phat}{P_1}) \lesssim \lambda \norm{H}_F.
	\]
	Finally, one can show that $\norm{\Phat}_F \leq \norm{\Phat}_* \leq 1$ implies $\norm{H}_F^2 \lesssim 1 - \ipHS{\Phat}{P_1}$, which immediately gives the result. 
	
	\section{Computational limitations and a practical algorithm for phase retrieval}
	\label{sec:opt}
	Although the mixed atomic norm $\normMixeds{\cdot}$ is a powerful theoretical tool,
	it is not clear how to calculate (let alone optimize) it for a general matrix in practice,
	since it is defined as an infimum over infinite sets of possible factorizations.
	
	A warning that computations with these atomic regularizers may be difficult in general is that they can be used to get $O_{\log}(s)$ sample complexity for sparse PCA,
	which, as discussed in \Cref{sec:pca_intro}, is widely believed to be impossible with efficient algorithms.

	In this section, we will analyze the convex programs more carefully,
	with a particular focus on phase retrieval.%
	\footnote{While our algorithmic approach led to strong empirical performance for sparse phase retrieval, the approach was less effective for sparse PCA. We leave a more thorough investigation of this phenomenon for future work.}
	We will analyze the optimality conditions via a dual problem and thereby develop a heuristic algorithm.
	
	This problem was studied in greater generality in \cite{Haeffele2020}. Their Corollary 1 is similar to our \Cref{cor:opt_factored}.
	However, our analysis of the dual problem is quite different from their perturbation argument,
	and we can much more easily apply our method to the sparse PCA optimization problem \eqref{eq:pca_opt} with its inequality constraint.
	Furthermore, we think the reader will benefit from our deriving the optimality conditions from more elementary principles for the particular problem we are trying to solve.

	\subsection{Factorization, duality, and optimality conditions}
	\label{sec:optimality}
	To move toward a practical algorithm, we consider optimizing \eqref{eq:pr_opt_main} in factored form;
	rather than optimizing over $B$ directly, we optimize over the factors $\{u_k, v_k\}$ of a factorization $B = \sum_k u_k \otimes v_k$.
	Then \eqref{eq:pr_opt_main} is equivalent to
	\begin{equation}
		\label{eq:pr_opt_factored}
		\min_{\{u_k, v_k\} \subset \R^p}~\frac{1}{2n} \sum_{i=1}^n \parens*{y_i - \ipHS*{X_i}{\sum_k u_k \otimes v_k}}^2 + \lambda \sum_k \theta_s(u_k, v_k).
	\end{equation}
	The obvious drawback to this form is that the optimization problem is no longer convex;
	therefore, it is not clear whether finding a global minimum is computationally feasible.

	To determine how well a factored algorithm works (e.g., to certify optimality),
	we examine a dual problem to \eqref{eq:pr_opt_main}.
	We formulate the dual via a trick found in \cite{Zhang2002}:
	note that $b^2/2 = \max_a~ab - a^2/2$ (achieved if and only if $a = b$),
	and therefore
	\begin{align*}
		&\min_{B \in \R^{p \times p}}~\frac{1}{2n} \sum_{i=1}^n (y_i - \ipHS{X_i}{B})^2 + \lambda \normMixeds{B} \\
		&\qquad= \min_{B \in \R^{p \times p}}~\frac{1}{2n} \sum_{i=1}^n \max_{\alpha_i}~\parens*{ 2\alpha_i (y_i - \ipHS{X_i}{B}) - \alpha_i^2} + \lambda  \normMixeds{B} \\
		&\qquad\geq \max_{\alpha \in \R^n}~\brackets*{  \frac{1}{n} \sum_{i=1}^n  \parens*{\alpha_i y_i - \frac{\alpha_i^2}{2}} + \min_{B \in \R^{p \times p}}~\parens*{ \lambda  \normMixeds{B} - \frac{1}{n} \sum_{i=1}^n \alpha_i \ipHS{X_i}{B}} },
	\end{align*}
	where the inequality comes from swapping the maximum over $\alpha = (\alpha_1, \dots, \alpha_n)$ and the minimum over $B$.
	
	Define the dual norm $\normMixeds{\cdot}^*$ by
	\[
		\normMixeds{Z}^* = \max_{\substack{B \in \R^{p \times p} \\ \normMixeds{B} \leq 1}}~\ipHS{Z}{B}.
	\]
	Because $\normMixeds{\cdot}^*$ is nonnegatively homogeneous,
	\[
		\min_{B \in \R^{p \times p}}~\parens*{ \lambda  \normMixeds{B} - \ipHS*{\frac{1}{n} \sum_{i=1}^n \alpha_i X_i}{B} }
		= \begin{cases}
			0 & \text{if } \normMixeds*{ \frac{1}{n} \sum_{i=1}^n \alpha_i X_i }^* \leq \lambda \\
			-\infty & \text{otherwise}.
		\end{cases}
	\]
	Therefore, a dual formulation of \eqref{eq:pr_opt_main} is the convex problem
	\begin{equation}
		\label{eq:pr_dual}
		\max_{\alpha \in \R^n}~\parens*{\frac{1}{n} \sum_{i=1}^n  \alpha_i y_i - \frac{\alpha_i^2}{2}}~\mathrm{s.t.}~\normMixeds*{ \frac{1}{n} \sum_{i=1}^n \alpha_i X_i }^* \leq \lambda.
	\end{equation}
	
	Before we go further, note that,
	\[
		\normMixeds{Z}^* = \max_{\substack{u, v \in \R^p\\ \theta_s(u, v) \leq 1}}~\ip{Z u}{v}.
	\]
	To see this, note that
	\begin{align*}
		\normMixeds{Z}^*
		&= \sup~\braces*{\ipHS{Z}{B} : B \in \R^{p \times p}, \{u_k, v_k\} \subset \R^p, B = \sum_k u_k \otimes v_k,\ \sum_k \theta_s(u_k, v_k) \leq 1} \\
		&= \sup~\braces*{ \sum_k \ip{Z u_k}{v_k} : \{u_k, v_k\} \subset \R^p, \sum_k \theta_s(u_k, v_k) \leq 1 }\\
		&= \sup~\braces*{ \sum_{k=1}^K \ip{Z u_k}{v_k}: K \geq 1, \{u_k, v_k\}_{k=1}^K \subset \R^p, \sum_{k=1}^K \theta_s(u_k, v_k) \leq 1 }.
	\end{align*}
	For any finite sequence $\{ u_k, v_k \}_{k=1}^K$ with $\sum_{k=1}^K \theta_s(u_k, v_k) \leq 1$,
	if we let $k^* = \argmax_{1 \leq k \leq K}~\frac{\ip{Z u_k}{v_k}}{\theta_s(u_k, v_k)}$
	and set $\utl = \frac{u_{k^*}}{\sqrt{\theta_s(u_{k^*}, v_{k^*})}}$ and $\vtl = \frac{v_{k^*}}{\sqrt{\theta_s(u_{k^*}, v_{k^*})}}$,
	we will always have $\ip{Z \utl}{\vtl} \geq \sum_{k=1}^K \ip{Z u_k}{u_k}$.
	Therefore,
	\[
		\normMixeds{Z}^* = \sup~\{ \ip{Zu}{v} : \theta_s(u,v) \leq 1 \}.
	\]
	We can replace the supremum by a maximum because the objective function is continuous and the constraint set is compact.
	
	Returning to the optimization problem,
	note that a feasible point $\alpha$ for the dual problem gives us a \emph{lower} bound on the primal optimal value.
	If there exist $B \in \R^{p \times p}$, $\alpha \in \R^n$ such that $\alpha$ is feasible and the two objective functions are \emph{equal}, then we know $B$ is optimal for the primal problem.
	More precisely, $(B, \alpha)$ is an optimal primal-dual pair if and only if
	\begin{enumerate}[label=(\alph*)]
		\item the primal objective function at $B$ equals the dual objective functions at $\alpha$, and
		\item $\alpha$ is feasible, i.e., $\normMixeds*{ \frac{1}{n} \sum_{i=1}^n \alpha_i X_i }^* \leq \lambda$.
	\end{enumerate}
	From the derivation of the dual problem above, (a) requires $\alpha_i = y_i - \ipHS{X_i}{B}$.
	Making this substitution, setting the objective functions equal, and simplifying gives one direction of the following result:
	\begin{lemma}
		\label{lem:optimal}
		$B$ solves \eqref{eq:pr_opt_main} if and only if both of the following hold:
		\begin{enumerate}[label=(\alph*)]
			\item $\frac{1}{n} \sum_{i=1}^n (y_i - \ipHS{X_i}{B}) \ipHS{X_i}{B} = \lambda \normMixeds{B}$.
			\item $\normMixeds*{ \frac{1}{n} \sum_{i=1}^n (y_i - \ipHS{X_i}{B}) X_i }^* \leq \lambda$.
		\end{enumerate}
	\end{lemma}
	\begin{proof}
	We have already shown that these conditions are \emph{sufficient} for optimality.
	To see the other direction (that these conditions are \emph{necessary} for optimality),
	note that $Z \coloneqq \frac{1}{n} \sum_{i=1}^n (y_i - \ipHS{X_i}{B}) X_i$ is the negative gradient of the empirical loss at $B$.
	Because condition (b) is equivalent to
	\[
		\ip{Z u}{v} \leq \lambda \theta_s(u, v)\ \forall u \in \R^p,
	\]
	if (b) does not hold, there exists some $\ubr, \vbr \in \R^p$ such that $\ip{Z \ubr}{\vbr} > \lambda \theta_s(\ubr, \vbr)$,
	and then we can decrease the objective function by moving to $B + \epsilon \ubr \otimes \vbr$ for some sufficiently small $\epsilon > 0$. Thus (b) is a necessary condition for the optimality of $B$.
	
	Now suppose (b) holds, but (a) does not.
	Condition (b) implies that $\ipHS{Z}{B} \leq \lambda \normMixeds{B}$,
	so we must have $\ipHS{Z}{B} < \lambda \normMixeds{B}$.
	
	Let $B = \sum_k u_k \otimes v_k$ be an optimal factorization with respect to the definition of $\normMixeds{B}$, that is, such that $\normMixeds{B} = \sum_k \theta_s(u_k, v_k)$ (we assume, for clarity, that an optimal factorization exists---if not, we could use an approximation argument).
	There must be some $u_k, v_k$ such that $\ip{Z u_k}{v_k} < \lambda \theta_s(u_k, v_k)$.
	Then, modifying $B$ by replacing $(u_k, v_k)$ with $((1-\epsilon) u_k, (1-\epsilon) v_k)$ for some sufficiently small $\epsilon > 0$ will decrease the objective function.
	
	\end{proof}
	Note that the proof of \Cref{lem:optimal} gives us an explicit way to improve the objective function whenever one of the optimality conditions is not satisfied.
	
	Applying our derivation to the factored optimization problem,
	we get the following result:
	\begin{corollary}
		\label{cor:opt_factored}
		$B$ solves \eqref{eq:pr_opt_main} and $B = \sum_k u_k \otimes v_k$ is an optimal factorization with respect to $\normMixeds{\cdot}$ (equivalently, $\{u_k, v_k\}$ solve \eqref{eq:pr_opt_factored})
		if and only if the following hold:
		\begin{enumerate}[label=(\alph*)]
			\item For all $k$, $\frac{1}{n} \sum_{i=1}^n (y_i - \ipHS{X_i}{B}) \ip{X_i u_k}{v_k} = \lambda \theta_s(u_k, v_k)$.
			\item $\normMixeds*{ \frac{1}{n} \sum_{i=1}^n (y_i - \ipHS{X_i}{B}) X_i }^* \leq \lambda$; equivalently, for all $u,v \in \R^p$,
			\[
				\frac{1}{n} \sum_{i=1}^n (y_i - \ipHS{X_i}{B}) \ip{X_i u}{v} \leq \lambda \theta_s(u, v).
			\]
		\end{enumerate}
	\end{corollary}
	Note that we have broken out condition (a) into individual equalities (rather than equating the sums of each side); condition (b) allows us to do this.
	It is even easier to find a descent direction when one of these conditions fails to hold,
	since the objective function of \eqref{eq:pr_opt_factored}
	already depends explicitly on the vectors $u_k, v_k$.
	
	Note that condition (a) is much easier to verify than condition (b).
	We refer to $\{ u_k, v_k \}$ as a \emph{first-order stationary point} if it satisfies condition (a),
	since this is equivalent to a zero subgradient on the (nonzero) $u_k$'s and $v_k$'s
	(cf.\ Proposition 2 in \cite{Haeffele2020}).
	
	Although we are not focusing on sparse PCA here,
	it may be interesting to compare \Cref{cor:opt_factored} to what we get for sparse PCA,
	particularly as PCA may be a fundamentally more difficult problem.
	A dual problem to \eqref{eq:pca_opt} is
	\[
		\argmax_{Z \in \R^{p \times p}}~ - \norm{Z}~\text{s.t}~\normMixeds{\Sigmahat - Z}^* \leq \lambda.
	\]
	The following lemma gives (redundant) optimality conditions:
	\begin{lemma}
		\label{lem:opt_pca}
		$P$ solves \eqref{eq:pca_opt} if and only if $\norm{P}_* = 1$ and there exists $Z \in \R^{p \times p}$ such that
		\begin{enumerate}
			\item $\normMixeds{\Sigmahat - Z}^* \leq \lambda$,
			\item $\ipHS{\Sigmahat - Z}{P} = \lambda \normMixeds{P}$,
			\item $\ipHS{Z}{P} = \norm{Z} = \norm{Z} \norm{P}_*$, and
			\item $\norm{Z} = \ipHS{\Sigmahat}{P} - \lambda \normMixeds{P}$.
		\end{enumerate}
	\end{lemma}
	In the PCA case, the semidefinite version of the problem is somewhat simpler due to the fact that the nuclear norm becomes a trace.
	If we solve
	\[
		\Phat = \argmin_{P \succeq 0}~ - \ipHS{\Sigmahat}{P} + \lambda \Theta_s(P)~\text{s.t.}~\tr(P) \leq 1,
	\]
	we get similar theoretical error guarantees as \Cref{thm:pca}.
	Furthermore, $P = \sum_k u_k \otimes u_k$ solves this optimization program and $\{u_k\}$ is an optimal factorization with respect to $\Theta_s$ if and only if $P$ is feasible and,
	for all $u \in \R^p$.
	\[
		\ip{\Sigmahat u}{u} + \parens*{ \lambda \sum_k \theta_s(u_k, u_k) - \ipHS{\Sigmahat}{P}} \norm{u}_2^2 \leq \theta_s(u, u).
	\]
	
	\subsection{A first factored algorithm, a computational snag, and a heuristic}
	\label{sec:practical_alg}
	The results of the previous section give a simple abstract recipe for finding a global optimum of \eqref{eq:pr_opt_main}:
	\begin{enumerate}
		\item We optimize \eqref{eq:pr_opt_factored} over a fixed number $r$ of rank-1 factors (i.e., vectors $u_1, \dots, u_r, v_1, \dots, v_r$) until we reach a first-order stationary point (by satisfying condition (a) in \Cref{cor:opt_factored}).
		Note that whenever condition (a) is not satisfied, it is easy to find a descent direction,
		since we can simply rescale the vectors $u_k, v_k$ in a similar manner to the second part of the proof of \Cref{lem:optimal}.
		\item At a first-order stationary point, if condition (b) in \Cref{cor:opt_factored} holds, we have reached the global minimum.
		Otherwise, as in the first part of the proof of \Cref{lem:optimal}, there exists $\utl, \vtl \in \R^p$ such that $\frac{1}{n} \sum_{i=1}^n (y_i - \ipHS{X_i}{B}) \ip{X_i \utl}{\vtl} > \lambda \theta_s(\utl, \vtl)$.
		We set $(u_{r+1}, v_{r+1}) = (\epsilon \utl, \epsilon \vtl)$ for $\epsilon > 0$ small enough to decrease the objective function and go back to step 1.
	\end{enumerate}
	The algorithm is guaranteed to terminate with a finite $r$ by \cite[Theorem 2]{Haeffele2020}.

	The most difficult part to implement is step 2.
	Checking condition (b) requires maximizing a bilinear form on vectors $u, v$ under a bound on $\theta_s(u, v)$.
	If we could maximize this for general bilinear forms, we could also solve sparse PCA (see \Cref{sec:optimality}), so we suspect it is not possible.
	However, this does not preclude positive results that exploit the particular structure of the phase retrieval problem.

	To implement a practical algorithm,
	we take a very simple shortcut:
	instead of checking condition (b) over \emph{all} vectors $u, v \in \R^p$,
	we check it over \emph{1-sparse} vectors.
	We simply calculate whether any element of $\frac{1}{n} \sum_{i=1}^n (y_i - \ipHS{X_i}{B}) X_i$
	is greater than $(1 + 1/\sqrt{s})^2\lambda$.
	Although we have not yet found a robust theoretical justification,
	we will see in the next section that this heuristic works reasonably well in practice.
	We summarize our high-level practical algorithm in \Cref{alg:spr_hilevel}.

	\begin{algorithm}
		\caption{High-level sparse phase retrieval algorithm}
		\label{alg:spr_hilevel}
		\begin{algorithmic}[1]
			\State $r \gets 1$
			\State Initialize $u_1, v_1$ (e.g., some spectral algorithm)
			\While{not Converged}
				\State Optimize \eqref{eq:pr_opt_factored} over $\{u_1, \dots u_r\}, \{ v_1, \dots, v_r \}$ with first-order method until condition (a) in \Cref{cor:opt_factored} is satisfied
				\State $Z \gets \frac{1}{n} \sum_{i=1}^n (y_i - \ipHS{X_i}{B}) X_i$, where $B = \sum_{k=1}^r u_k \otimes v_k$
				\If{$Z_{ij} > (1 + 1/\sqrt{s})^2\lambda$ for any $i,j \in \{1, \dots, p\}$}
					\State $r \gets r + 1$
					\State $u_{r+1} \gets \epsilon e_j, v_{r+1} \gets \epsilon e_i$, where $\epsilon > 0$ is sufficiently small to decrease objective function.
				\Else
					\State Converged $\gets$ true
				\EndIf
			\EndWhile
			\State \Return $\{u_1, \dots, u_r\}, \{ v_1, \dots, v_r \}$
		\end{algorithmic}
	\end{algorithm}
	
	\begin{figure}
		\centering
		\begin{subfigure}[t]{0.45\textwidth}
			\centering
			\includegraphics[width=0.95\textwidth]{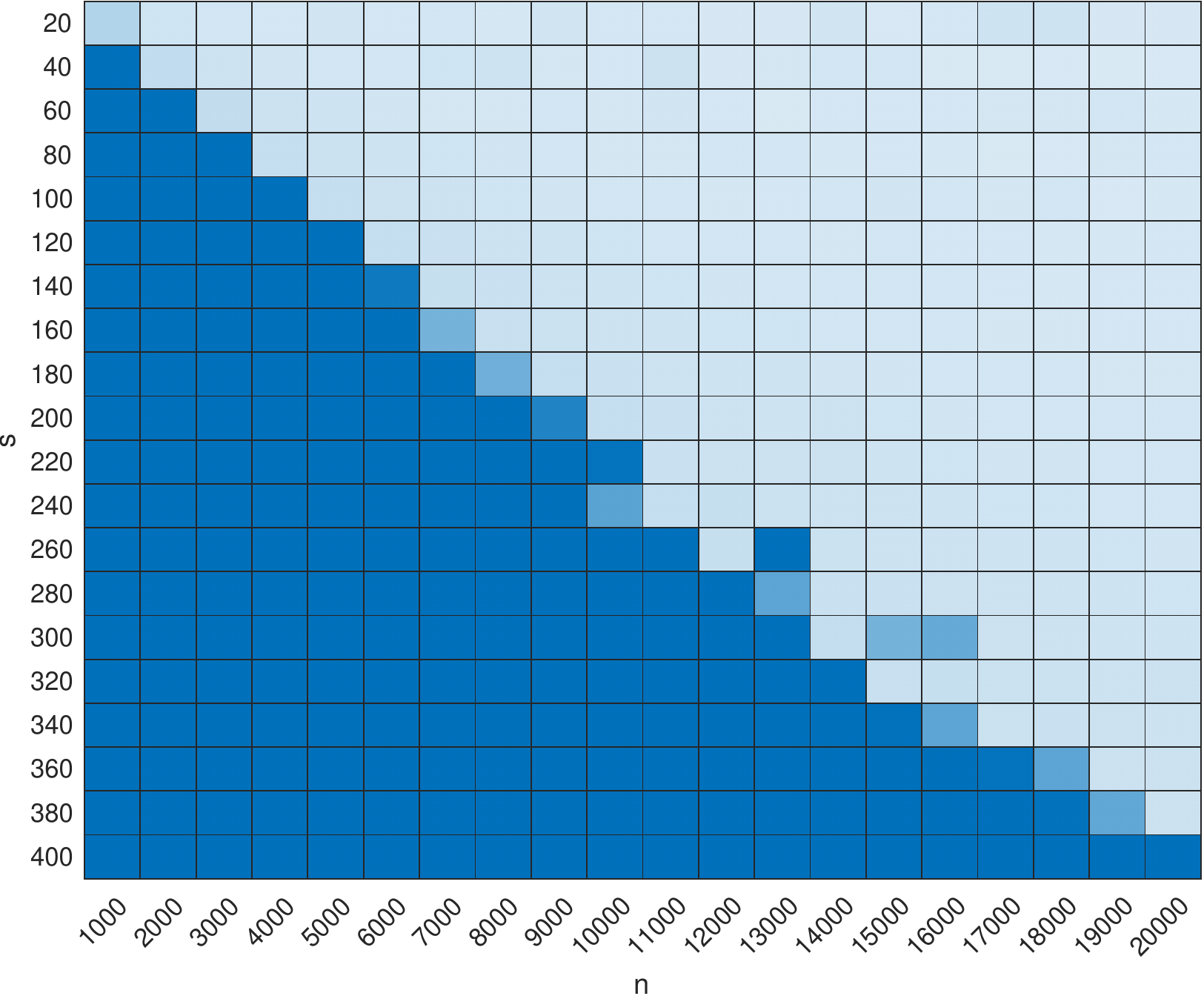}
			\subcaption{Our algorithm}
		\end{subfigure}
		\begin{subfigure}[t]{0.45\textwidth}
			\centering
			\includegraphics[width=0.95\textwidth]{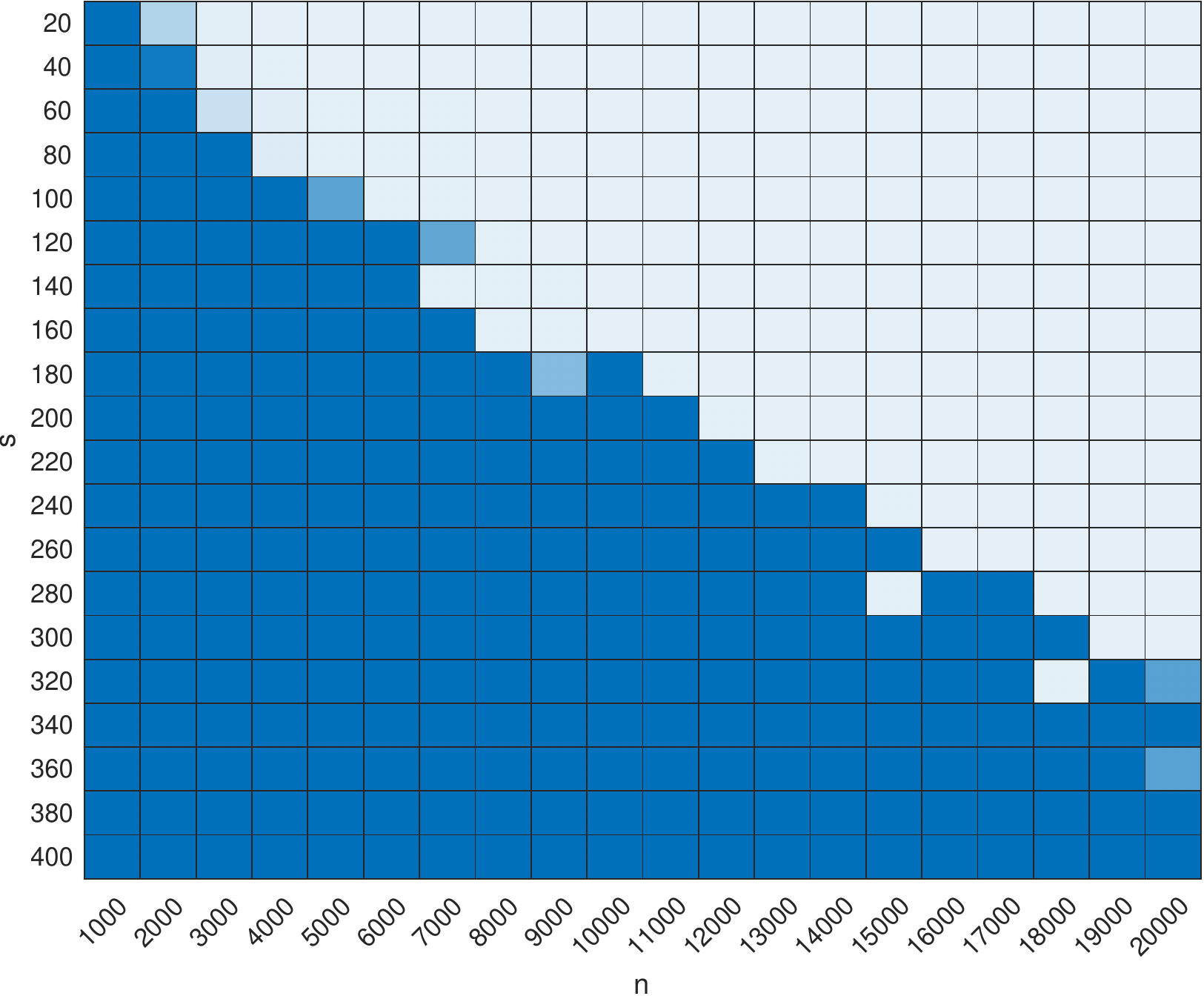}
			\subcaption{SWF \cite{Yuan2019}}
		\end{subfigure}\\[1em]
		\begin{subfigure}[t]{0.45\textwidth}
			\centering
			\includegraphics[width=0.95\textwidth]{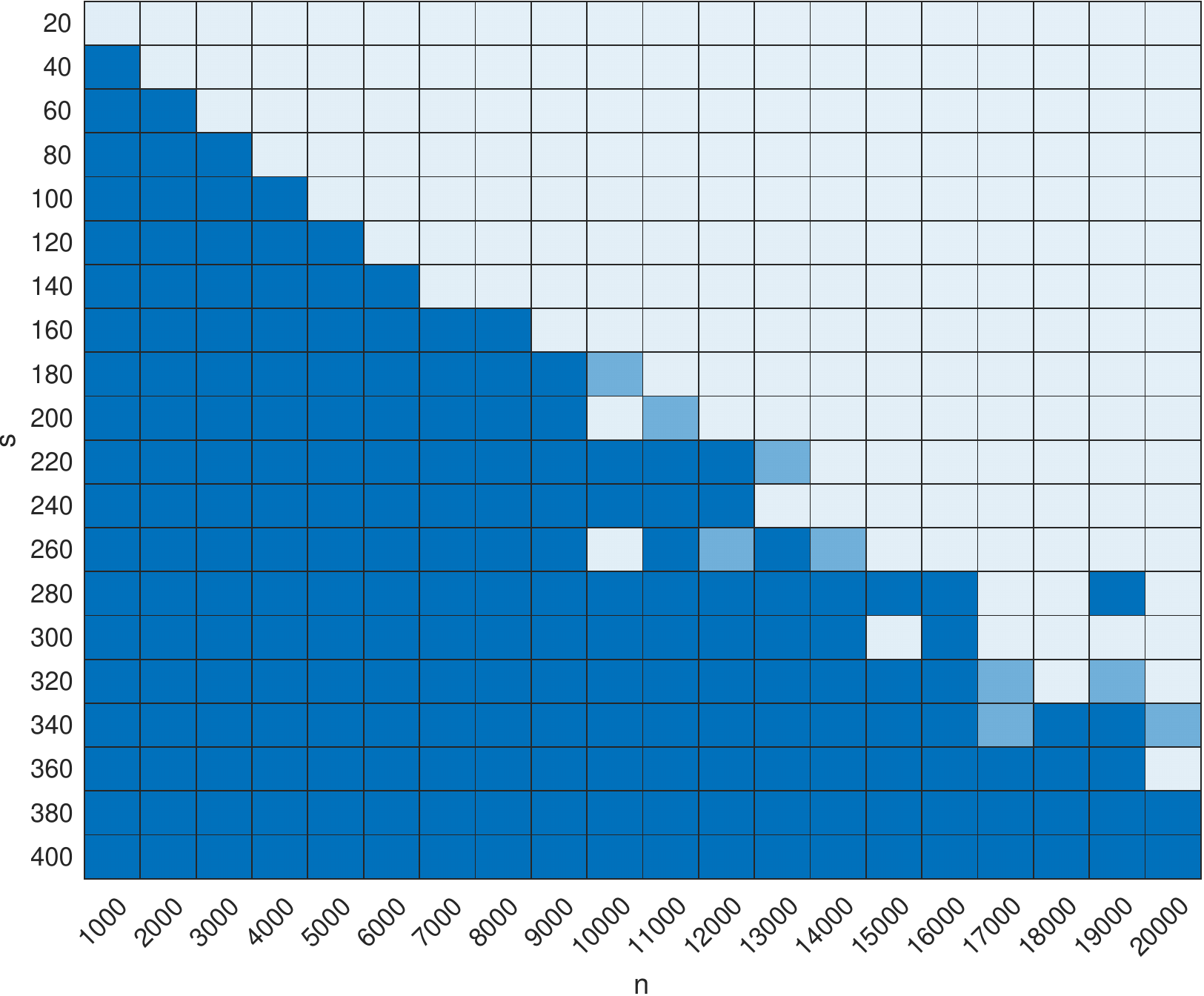}
			\subcaption{GAMP \cite{Schniter2015}}
		\end{subfigure}
		\begin{subfigure}[t]{0.45\textwidth}
			\centering
			\includegraphics[width=0.95\textwidth]{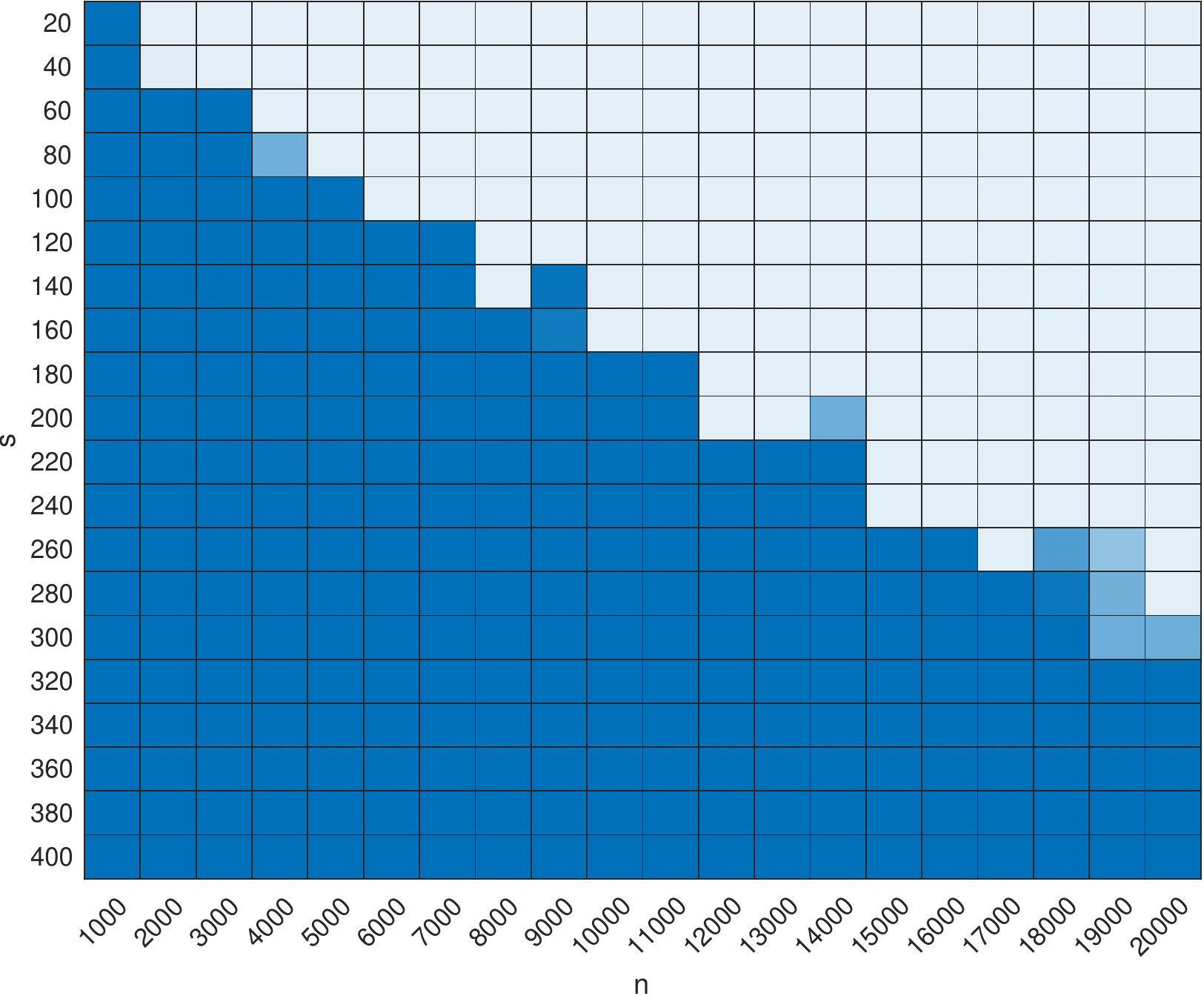}
			\subcaption{SPARTA \cite{Wang2018}}
		\end{subfigure}\\[1em]
		\begin{subfigure}[t]{0.45\textwidth}
			\centering
			\includegraphics[width=0.95\textwidth]{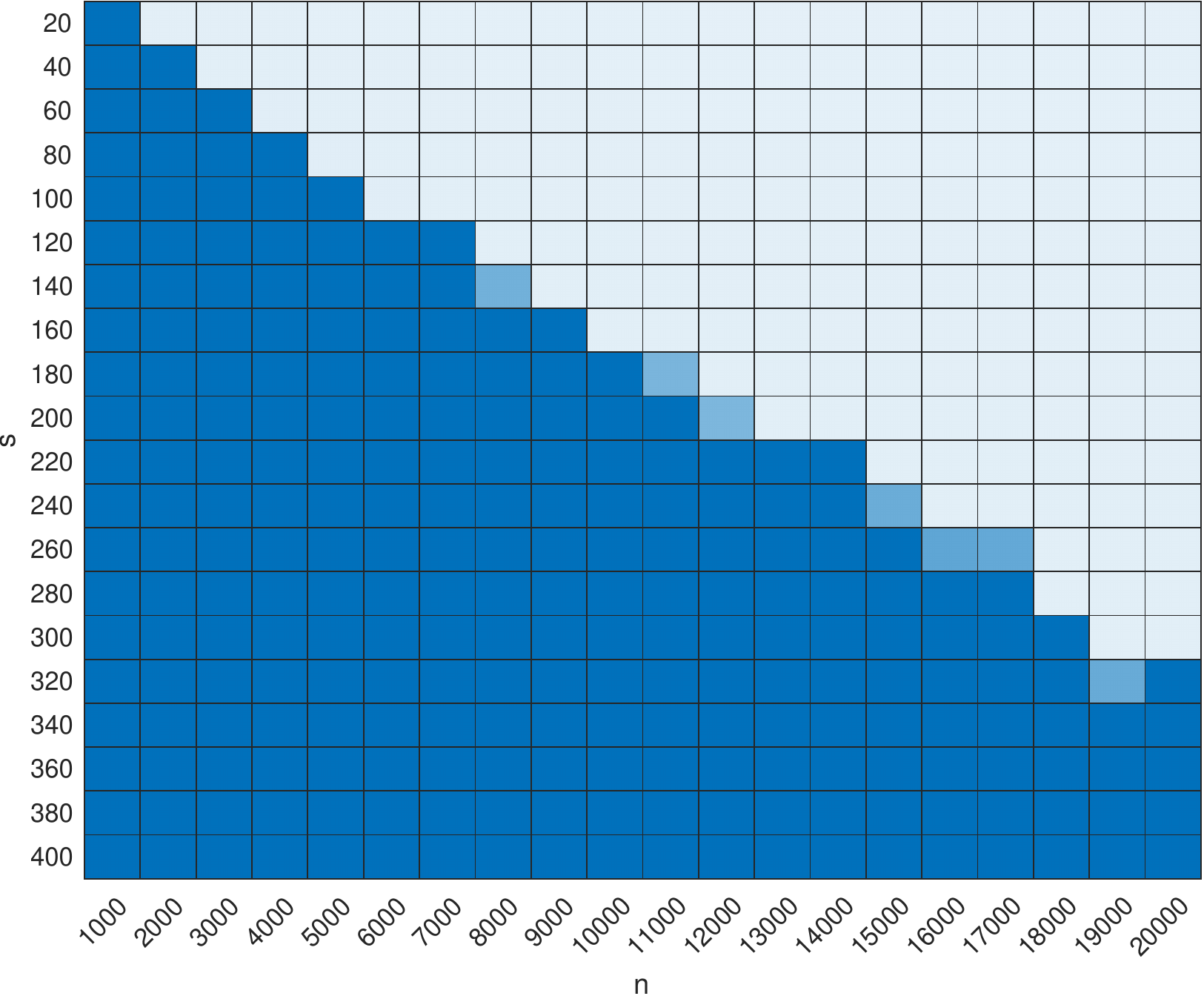}
			\subcaption{CoPRAM \cite{Jagatap2019}}
		\end{subfigure}
		\caption{Phase transition plots. Colors represent 80\% quantile error over 20 trials (darker colors correspond to higher error). We used $p = 20{,}000$, $\norm{\beta^*}_2 = 1$, and $\sigma = 0.05$. All algorithms were run on the same data.}
		\label{fig:phase_tr_sims}
	\end{figure}
	
	\begin{figure}
		\centering
		\begin{subfigure}[t]{0.49\textwidth}
			\centering
			\includegraphics[width=0.98\textwidth]{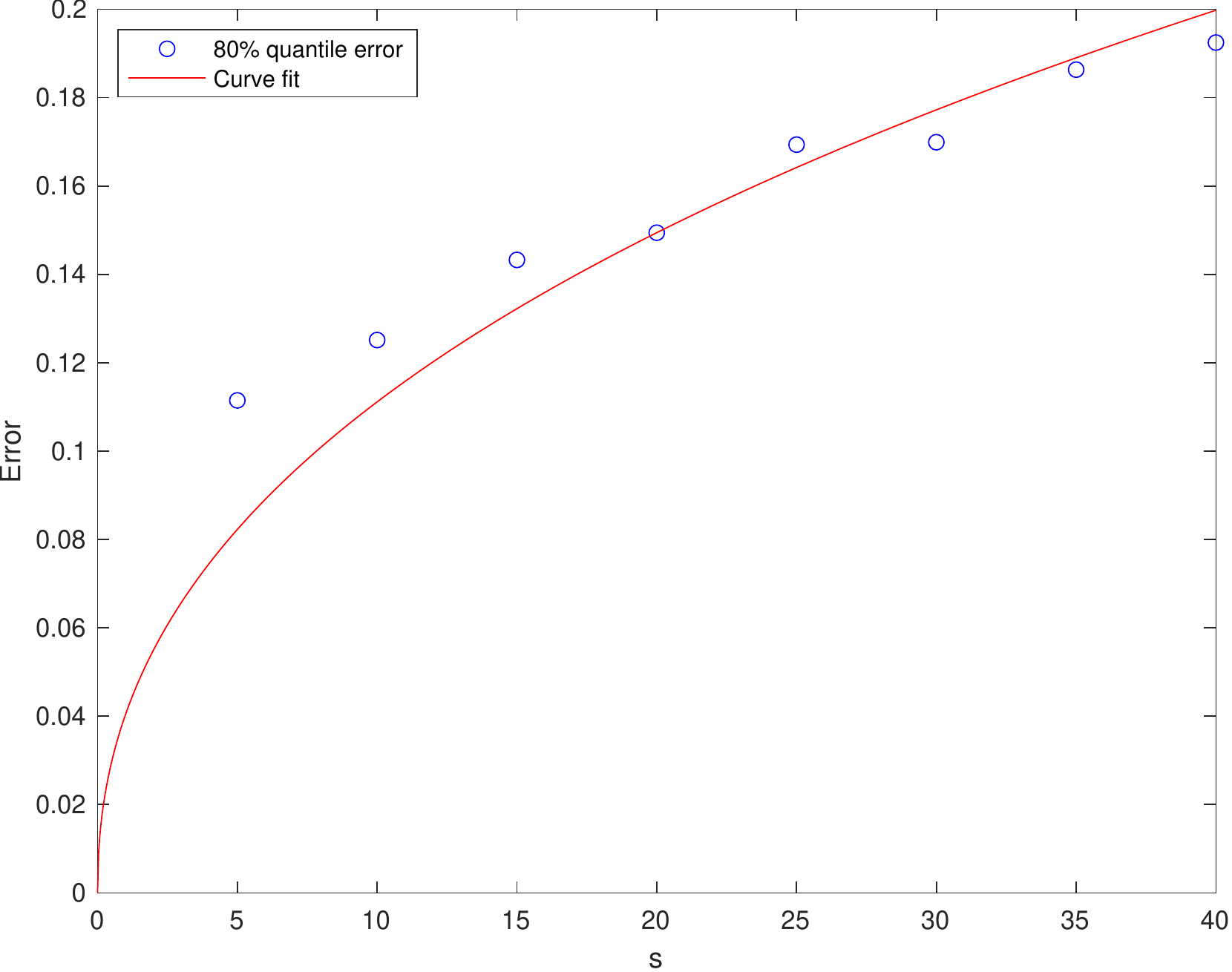}
			\subcaption{Gaussian noise ($\norm{\beta^*}_2 = 1$, $\sigma = 0.2$)}
		\end{subfigure}
		\begin{subfigure}[t]{0.49\textwidth}
			\centering
			\includegraphics[width=0.98\textwidth]{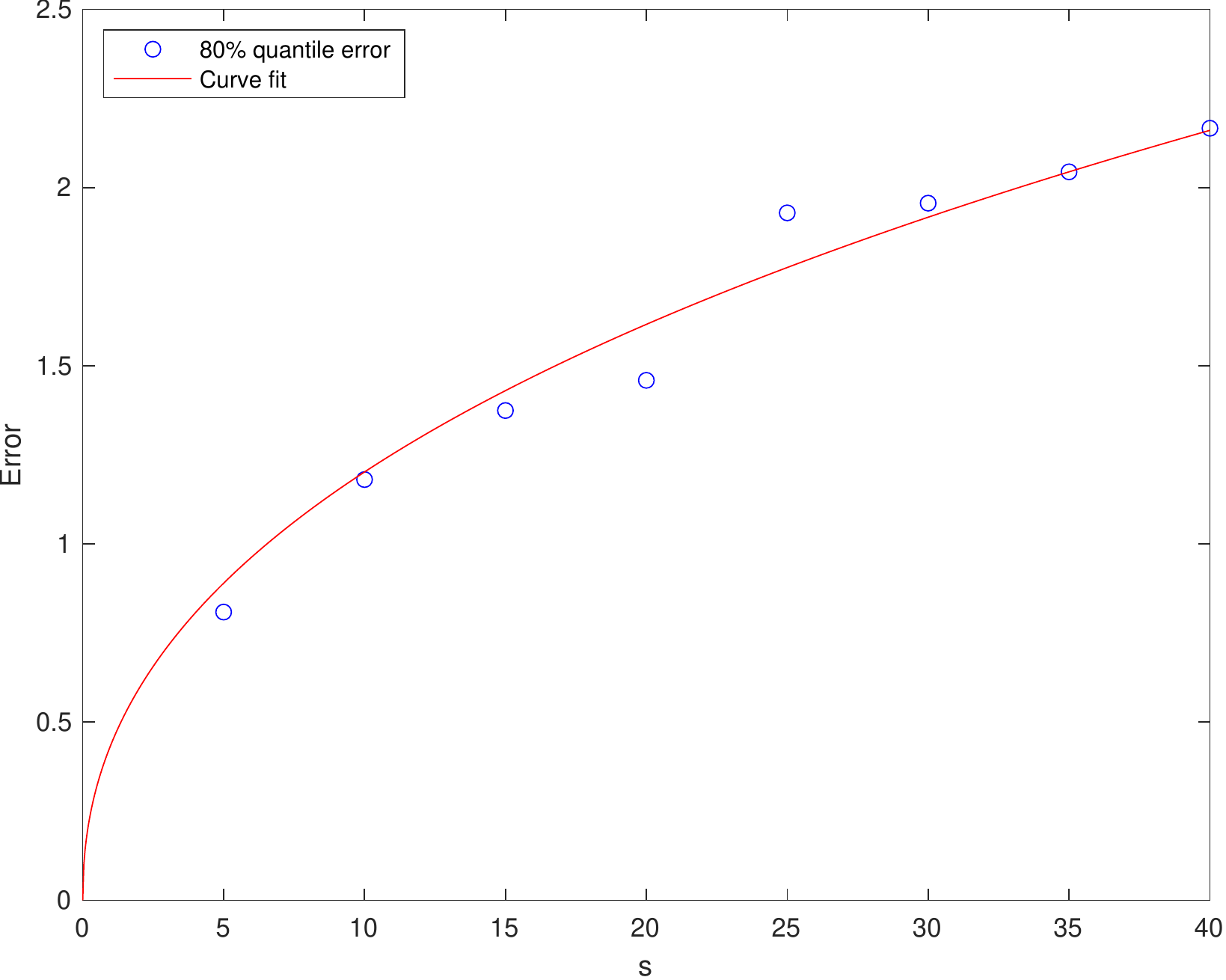}
			\subcaption{Poisson noise ($\norm{\beta^*}_2 = 10$)}
		\end{subfigure}
		\caption{Plot of $\norm{\betahat - \beta^*}_2$ vs.\ $s$ (80\% quantile over 10 trials). All simulations use $p = 8{,}000$ and $n = 4{,}000$.
			Blue circles are actual data; the red curves are of the form $c \sqrt{s \log \frac{ep}{s}}$, where the scaling factor $c$ is chosen to give minimum mean absolute deviation.}
		\label{fig:s_sweep_sim}
	\end{figure}
	
	\subsection{Simulation results}
	\label{sec:sims}
	We implemented \Cref{alg:spr_hilevel} in MATLAB and ran a variety of simulations to illustrate its performance with respect to both sample complexity and noise performance.
	The interested reader can view our code\footnote{\url{https://github.com/admcrae/spr2021}} to see more details, but some of the more salient features are the following:
	\begin{itemize}
		\item Line 5 of \Cref{alg:spr_hilevel} is implemented with alternating minimization over $U = [u_1 \cdots u_r] \in \R^{p \times r}$ and $V = [v_1 \cdots v_r] \in \R^{p \times r}$.
		\item After each alternating minimization step, we ``rebalance'' $U$ and $V$ (i.e., rescale each $u_k, v_k$ to force $\theta_s(u_k, u_k) = \theta_s(u_k, v_k) = \theta_s(v_k, v_k)$).
		\item Each minimization problem over $U$ or $V$ is convex, and we solve it with an accelerated proximal gradient descent algorithm. 
		\item The proximal step requires solving a convex problem of the form
		\[
			\argmin_{y \in \R^p}~\ip{x}{y} + \frac{1}{2} \norm{y}_2^2 + a \norm{y}_2 + b \norm{y}_1
		\]
		for arbitrary $x \in \R^p$ and $a, b > 0$.
		This can be solved in closed form by soft-thresholding $x$ with threshold $b$ and then rescaling.
	\end{itemize}
	
	All of our simulations used i.i.d.\ Gaussian measurement vectors $x \sim \normaldist(0, I_p)$.
	\begin{enumerate}
		\item \Cref{fig:phase_tr_sims} shows phase transition diagrams of performance versus sample size $n$ and sparsity $s$ for our algorithm and a variety of alternatives. Note that qualitatively, all these algorithms have similar performance in terms of sample complexity.
		Interestingly, all of them appear only to require (within a log factor) a number of samples \emph{linear} in the sparsity $s$.
		This demonstrates a gap between the empirical performance of all these algorithms and the best theoretical guarantees that have been proved so far.
		
		\item \Cref{fig:s_sweep_sim} shows plots of the error versus sparsity $s$ for both Gaussian noise and Poisson noise.
		Note that in both cases, the error roughly follows the predicted $\sqrt{s \log(p/s)}$ scaling.
		
	\end{enumerate}
	
	\section{Conclusion}
	We have shown that estimators for sparse phase retrieval and sparse PCA obtained by solving a convex program (\eqref{eq:pr_opt_main} for sparse phase retrieval and \eqref{eq:pca_opt} for sparse PCA) with the abstract mixed atomic norm \eqref{eq:mixednorm_def} as a regularizer satisfy optimal statistical guarantees in terms of sample complexity and error.
	For sparse phase retrieval, we have derived a practical heuristic algorithm whose performance matches that of existing state-of-the-art algorithms.
	
	Our work suggests new methods for analyzing these problems (and others with similar sparse factored structure, such as sparse blind deconvolution).
	It also suggests interesting new research directions in sparse recovery and in optimization.
	For example, it would be very useful to study \emph{why} our heuristic approach appears to work well for sparse phase retrieval as well as whether it is possible to do even better. A related problem is to prove that sparse phase retrieval has linear sample complexity with practical algorithms (or that it doesn't, along with why current empirical results seem to suggest otherwise).
	Similarly, the atomic matrix norm (along with other similar norms) invites further analysis, particularly in how well we can optimize it (where this may depend on the structure of the problem in which it is used).
	The interplay between statistical guarantees and computational complexity theory (e.g., in sparse PCA) may be very interesting here.
	\clearpage
	\appendix
	\section{Detailed analysis of mixed norm}
\label{app:mixednorm}
In this section, we explore several important properties of the mixed norm $\normMixeds{\cdot}$.

First, we show that matrices with small mixed norm can be written as a convex combination of sparse rank-1 matrices.
\begin{lemma}
	\label{lem:norm_atomic_equiv}
	For any matrix $A$, we can write $A = \sum a_i u_i \otimes v_i$,
	where each $u_i$ and $v_i$ has unit $\ell_2$ norm and is $s$-sparse, and $\sum \abs{a_i} \leq \normMixeds{A}$.
	
	Consequently, for any matrix $Z$,
	\[
		\sup_{\normMixeds{A} \leq 1} \ipHS{Z}{A}
		\leq \sup_{\substack {\norm{u}_2, \norm{v}_2 \leq 1 \\ \norm{u}_0, \norm{v}_0 \leq s}} \ip{Z u}{v}.
	\]
\end{lemma}
\begin{proof}
	The consequence follows from the first statement immediately by the fact that any unit-atomic-norm $A$ is in the convex hull of rank-1 $s$-sparse atoms. We now prove the first statement of the Lemma.
	
	Because $\normMixeds{\cdot}$ is defined as an atomic norm over rank-1 atoms,
	it suffices to prove the result for rank-1 $A$.
	Therefore, we will show that any rank-1 matrix $x \otimes y$ can be written as $x \otimes y = \sum u_i \otimes v_i$,
	where each $u_i$ and $v_i$ is $s$-sparse,
	and $\sum \norm{u_i}_2 \norm{v_i}_2 \leq \theta_s(x, y)$.
	
	Indeed, a standard result from sparsity theory (see, e.g., Exercise 10.3.7 in \cite{Vershynin2018}) says that any vector $z$ can be written as $z = \sum z_i$,
	where each $z_i$ is $s$-sparse, and $\sum \norm{z_i}_2 \leq \norm{z}_2 + \frac{1}{\sqrt{s}} \norm{z}_1$.
	Applying this to both $x$ and $y$,
	we have
	\[
	x \otimes y = \parens*{ \sum_i x_i } \parens*{ \sum_j y_j}
	= \sum_{i,j} x_i \otimes y_j,
	\]
	where each $x_i$ and $y_j$ is $s$-sparse, and
	\[
	\sum_{i,j} \norm{x_i}_2 \norm{y_j}_2
	= \parens*{ \sum_i \norm{x_i}_2 } \parens*{ \sum_j \norm{y_j}_2 }
	\leq \parens*{\norm{x}_2 + \frac{\norm{x}_1}{\sqrt{s}}} \parens*{\norm{y}_2 + \frac{\norm{y}_1}{\sqrt{s}}}
	=\theta_s(x, y).
	\]
\end{proof}

To prove \Cref{lem:subgrad_ineq},
we need to find a suitable subgradient of $\normMixeds{\cdot}$ at the point $B = \beta \otimes \beta$.
Let $I \subset \{1, \dots, p\}$ denote the indices for which the entries of $\beta$ are nonzero.
With some abuse of notation, we also write $I$ as the subspace of $\R^{p \times p}$ whose matrices are zero except at entries $(i,j) \in I \times I$.
We also denote $T = \{ x \otimes \beta + \beta \otimes y : x, y \in \R^p \}$.
We will denote the orthogonal projections onto these subspaces and various orthogonal complements and intersections by $\projI$, $\projT$, $\projTIp$, etc.
We will also on occasion denote the orthogonal projection onto $\spn\{\beta\} \subset \R^p$ or its orthogonal complement (in $I$) by $\projb$, $\projbp$, $\projbpI$, etc.

According to \cite[Proposition 1]{Haeffele2020}, a matrix $W \in \partial\normMixeds{B}$ if the following two properties hold:
\begin{enumerate}
	\item $\ip{W \beta}{\beta} = \theta_s(\beta, \beta)$, and
	\item $\ip{W u}{v} \leq \theta_s(u,v)$ for all $u, v \in \R^p$.
\end{enumerate}
It is easy to check that the matrix $W_\beta \coloneqq w_\beta \otimes w_\beta$,
where $w_\beta \coloneqq \frac{\beta}{\norm{\beta}_2} + \frac{1}{\sqrt{s}} \sign \beta$, is a subgradient.
However, as with the subgradients of the ordinary nuclear norm,
a much broader set of matrices satisfies these properties:
\begin{lemma}
	\label{lem:subgrads}
	Suppose $\beta$ is $s$-sparse, and let $B = \beta \otimes \beta$.
	Any matrix of the form $W = W_\beta + W^\perp \in \partial \normMixeds{B}$
	where $W^\perp$ can be any matrix in one of the following three families (or any convex combination thereof):
	\begin{enumerate}
		\item $W^\perp = \frac{1}{\sqrt{s}} (w_\beta \otimes \utl + \vtl \otimes w_\beta)$, where $\utl,\vtl \in I^\perp$ and $\norm{\utl}_\infty, \norm{\vtl}_\infty \leq 1$.
		\item $W^\perp \in T^\perp$ and $\norm{W} \leq 1$.
		\item $W^\perp = \projTpIp(\Wtl)$ for $\Wtl$ satisfying $\ip{\Wtl u}{v} \leq \frac{1}{5}\theta_s(u,v)$ for all $u,v \in \R^p$.
	\end{enumerate}
\end{lemma}
\begin{proof}
	For each case, note that $\ip{W^\perp \beta}{\beta} = 0$,
	so we only need to show that $\ip{W u}{v} \leq \theta_s(u, v)$ for all $u, v \in \R^p$.
	
	We will use the following simple fact many times: for any vector $u \in \R^p$,
	\[
		\abs{\ip{w_\beta}{u}} \leq \norm{\projb (u)}_2 + \frac{1}{\sqrt{s}} \norm{\projI(u)}_1.
	\]
	We prove each case separately.
	
	\paragraph{Case 1:} Let $\utl,\vtl \in I^\perp$ with $\norm{\utl}_\infty, \norm{\vtl}_\infty \leq 1$.
	Let
	\[
		W = w_\beta \otimes w_\beta + \frac{1}{\sqrt{s}} \parens*{ w_\beta \otimes \utl + \vtl \otimes w_\beta }.
	\]
	Then, for any $u,v \in \R^p$,
	\begin{align*}
		\ip{W u}{v}
		&= \ip{w_\beta}{u} \ip{w_\beta}{v} + \frac{1}{\sqrt{s}} \parens*{ \ip{w_\beta}{v}\ip{\utl}{u} + \ip{\vtl}{v} \ip{w_\beta}{u} } \\
		&\leq \parens*{ \norm{\projb(u)}_2 + \frac{1}{\sqrt{s}} \norm{\projI(u)}_1 } \parens*{ \norm{\projb(v)}_2 + \frac{1}{\sqrt{s}} \norm{\projI(v)}_1 } \\
		&\qquad+ \parens*{ \norm{\projb(v)}_2 + \frac{1}{\sqrt{s}} \norm{\projI(v)}_1 } \frac{\norm{\projIp(u)}_1}{\sqrt{s}} \\
		&\qquad+ \parens*{ \norm{\projb(u)}_2 + \frac{1}{\sqrt{s}} \norm{\projI(u)}_1 } \frac{\norm{\projIp(v)}_1}{\sqrt{s}} \\
		&\leq \parens*{ \norm{\projb(u)}_2 + \frac{1}{\sqrt{s}} \norm{u}_1 } \parens*{ \norm{\projb(v)}_2 + \frac{1}{\sqrt{s}} \norm{v}_1 } \\
		&\leq \theta_s(u, v),
	\end{align*}
	where the penultimate inequality uses the fact that $\norm{z}_1 = \norm{\projI(z)}_1 + \norm{\projIp(z)}_1$ for any vector $z$.
	
	\paragraph{Case 2:} Let $W^\perp \in T^\perp$ such that $\norm{W} \leq 1$.
	Let $u, v \in \R^p$. Note that $\ip{W^\perp u}{v} \leq \norm{\projbp(u)}_2 \norm{\projbp(v)}_2$.
	Then
	\begin{align*}
		\ip{W u}{v}
		&= \ip{w_\beta}{u} \ip{w_\beta}{v} + \ip{W^\perp u}{v} \\
		&\leq \parens*{ \norm{\projb(u)}_2 + \frac{1}{\sqrt{s}} \norm{\projI(u)}_1 } \parens*{ \norm{\projb(v)}_2 + \frac{1}{\sqrt{s}} \norm{\projI(v)}_1 } + \norm{\projbp(u)}_2 \norm{\projbp(v)}_2 \\
		&\leq \theta_s(u, v),
	\end{align*}
	where the last inequality uses that fact that
	\[
		\norm{\projb(u)}_2 \norm{\projb(v)}_2 + \norm{\projbp(u)}_2 \norm{\projbp(v)}_2
		\leq \norm{u}_2 \norm{v}_2.
	\]
	
	\paragraph{Case 3:} Let $\Wtl \in \R^{p \times p}$ satisfy $\ip{\Wtl u}{v} \leq \frac{1}{5}\theta_s(u, v)$.
	Let $W = W_\beta + \projTpIp(\Wtl)$.
	Then, for $u,v \in \R^p$,
	\begin{align*}
		\ip{W u}{v}
		&= \ip{w_\beta}{u} \ip{w_\beta}{v} + \ipHS{\projTpIp(\Wtl)}{v \otimes u} \\
		&= \ip{w_\beta}{u} \ip{w_\beta}{v} + \ipHS{\Wtl}{\projTpIp(v \otimes u)} \\
		&= \ip{w_\beta}{u} \ip{w_\beta}{v} + \ip{\Wtl \projIp(u)}{\projbpI(v)} + \ip{\Wtl \projbpI(u)}{\projIp(v)} + \ip{\Wtl \projIp(u)}{\projIp(v)} \\
		&\leq \parens*{ \norm{\projb(u)}_2 + \frac{1}{\sqrt{s}} \norm{\projI(u)}_1 } \parens*{ \norm{\projb(v)}_2 + \frac{1}{\sqrt{s}} \norm{\projI(v)}_1 } \\
		&\qquad+ \frac{1}{5}\parens*{ \norm{\projIp(u)}_2 + \frac{1}{\sqrt{s}} \norm{\projIp(u)}_1 } \parens*{ \norm{\projbpI(v)}_2 + \frac{1}{\sqrt{s}} \norm{\projbpI(v)}_1 } \\
		&\qquad+ \frac{1}{5}\parens*{ \norm{\projbpI(u)}_2 + \frac{1}{\sqrt{s}} \norm{\projbpI(u)}_1 } \parens*{ \norm{\projIp(v)}_2 + \frac{1}{\sqrt{s}} \norm{\projIp(v)}_1 } \\
		&\qquad+ \frac{1}{5}\parens*{ \norm{\projIp(u)}_2 + \frac{1}{\sqrt{s}} \norm{\projIp(u)}_1 } \parens*{ \norm{\projIp(v)}_2 + \frac{1}{\sqrt{s}} \norm{\projIp(v)}_1 } \\
		&\leq \parens*{ \norm{\projb(u)}_2 + \frac{1}{\sqrt{s}} \norm{\projI(u)}_1 } \parens*{ \norm{\projb(v)}_2 + \frac{1}{\sqrt{s}} \norm{\projI(v)}_1 } \\
		&\qquad+ \frac{2}{5} \parens*{ \norm{\projbp(u)}_2 + \frac{1}{\sqrt{s}} \norm{\projIp(u)}_1 } \norm{\projbp(v)}_2 \\
		&\qquad+ \frac{2}{5} \norm{\projbp(u)}_2 \parens*{ \norm{\projbp(v)}_2 + \frac{1}{\sqrt{s}} \norm{\projIp(v)}_1 } \\
		&\qquad+ \frac{1}{5} \parens*{ \norm{\projbp(u)}_2 + \frac{1}{\sqrt{s}} \norm{\projIp(u)}_1 } \parens*{ \norm{\projbp(v)}_2 + \frac{1}{\sqrt{s}} \norm{\projIp(v)}_1 } \\
		&\leq \parens*{ \norm{\projb(u)}_2 + \frac{1}{\sqrt{s}} \norm{\projI(u)}_1 } \parens*{ \norm{\projb(v)}_2 + \frac{1}{\sqrt{s}} \norm{\projI(v)}_1 } \\
		&\qquad + \parens*{ \norm{\projbp(u)}_2 + \frac{1}{\sqrt{s}} \norm{\projIp(u)}_1 } \parens*{ \norm{\projbp(v)}_2 + \frac{1}{\sqrt{s}} \norm{\projIp(v)}_1 }.
	\end{align*}
	To bound this last expression,
	we consider the terms that we get from multiplying everything out.
	Note again that
	\[
		\norm{\projb(u)}_2 \norm{\projb(v)}_2 + \norm{\projbp(u)}_2 \norm{\projbp(v)}_2 \leq \norm{u}_2 \norm{v}_2,
	\]
	and also 
	\[
		\norm{\projI(u))}_1 \norm{\projI(v)}_1 + \norm{\projIp(u)}_1 \norm{\projIp(v)}_1 \leq \norm{u}_1 \norm{v}_1.
	\]
	For the cross-terms, note that
	\begin{align*}
		\norm{\projb(u)}_2 \norm{\projI(v)}_1 + \norm{\projbp(u)}_2 \norm{\projIp(v)}_1
		&\leq \min_{c > 0} c \frac{ \norm{\projb(u)}_2^2 + \norm{\projbp(u)}_2^2 }{2} + \frac{1}{c} \frac{ \norm{\projI(v)}_1^2 + \norm{\projIp(v)}_1^2}{2 s} \\
		&\leq \min_{c > 0} \parens*{ c \frac{\norm{u}_2^2}{2} + \frac{1}{c} \frac{\norm{v}_1^2}{2s}} \\
		&= \frac{1}{\sqrt{s}} \norm{u}_2 \norm{v}_1.
	\end{align*}
	The similar inequality holds for $u$ and $v$ reversed.
	Therefore,
	\[
		\ip{W u}{v}
		\leq \parens*{ \norm{u}_2 + \frac{1}{\sqrt{s}} \norm{u}_1 } \parens*{ \norm{v}_2 + \frac{1}{\sqrt{s}} \norm{v}_1 }
		= \theta_s(u, v).
	\]
\end{proof}

With this, we can prove \Cref{lem:subgrad_ineq}.
\begin{proof}[Proof of \Cref{lem:subgrad_ineq}.]
	Let $A \in \R^{p \times p}$.
	We choose a subgradient $W \in \partial \normMixeds{B}$ as follows:
	Let
	\[
		W = W_\beta + \frac{1}{10}\parens*{ W_1^\perp + 4 W_2^\perp + 5 W_3^\perp },
	\]
	where we choose $W_i^\perp$, $i = 1,2,3$, as follows:
	\begin{enumerate}
		\item If $\projTIp(A) = \beta \otimes u + v \otimes \beta$ where $u,v \in I^\perp$,
		choose
		\[
			W_1^\perp = \frac{1}{\sqrt{s}} \parens*{ w_\beta \otimes \utl + \vtl \otimes w_\beta },
		\]
		where $\utl, \vtl \in I^\perp$, $\norm{\utl}_\infty, \norm{\vtl}_\infty \leq 1$ and $\ip{\utl}{u} = \norm{u}_1$, $\ip{\vtl}{v} = \norm{v}_1$.
		Then
		\begin{align*}
			\ip{W_1^\perp u}{v}
			&= \parens*{ \norm{\beta}_2 + \frac{\norm{\beta}_1}{\sqrt{s}} } \frac{\norm{u}_1 + \norm{v}_1}{\sqrt{s}} \\
			&\geq \theta_s(\beta, u) + \theta_s(\beta, v) - 2 \norm{\beta}_2 (\norm{u}_2 + \norm{v}_2) \\
			&\geq \normMixeds{\projTIp(A)} - 2\sqrt{2} \norm{\projTIp(A)}_F \\
			&\geq \normMixeds{\projTIp(A)} - 2\sqrt{2} \norm{A}_F.
		\end{align*}
	
		\item Choose $W_2^\perp \in T^\perp \cap I$ with $\norm{W_2^\perp} \leq 1$ such that $\ipHS{W_2^\perp}{A} = \norm{\projTpI(A)}_* \geq \frac{1}{4} \normMixeds{\projTpI(A)}$.
		This last norm inequality holds because every vector in $I$ is $s$-sparse.
		
		\item Choose $W_3^\perp$ according to \Cref{lem:subgrads} such that $\ipHS{W_3^\perp}{A} = \frac{1}{5} \normMixeds{\projTpIp(A)}$.
	\end{enumerate}
	Then, using the fact that $\norm{W_\beta}_F = \norm{w_\beta}_2^2 \leq 4$, we have
	\begin{align*}
		\ipHS{W}{A}
		&= \ipHS{W_\beta}{A} + \frac{1}{10}\ipHS{W_1^\perp}{A} + \frac{4}{10} \ipHS{W_2^\perp}{A} + \frac{5}{10}\ipHS{W_3^\perp}{A} \\
		&\geq -4 \norm{A}_F - \frac{1}{10} \normMixeds{\projTI(A)} + \frac{1}{10} \normMixeds{\projTI(A)}\\
		&\qquad+ \frac{1}{10} \parens*{ \normMixeds{\projTIp(A)} - 2\sqrt{2} \norm{A}_F } + \frac{4}{10} \cdot \frac{1}{4} \normMixeds{\projTpI(A)} + \frac{5}{10} \cdot \frac{1}{5} \normMixeds{\projTpIp(A)} \\
		&\geq \frac{1}{10} \normMixeds{A} - \parens*{4 + \frac{\sqrt{2}}{5}} \norm{A}_F - \frac{1}{10} \normMixeds{\projTI(A)} \\
		&\geq \frac{1}{10} \normMixeds{A} - 5 \norm{A}_F,
	\end{align*}
	where the last inequality uses the fact that $\normMixeds{\projTI(A)} \leq 4 \norm{\projTI(A)}_* \leq 4 \sqrt{2} \norm{A}_F$.
\end{proof}

\section{Empirical process and restricted lower isometry bounds}
\label{app:empirical_proofs}

\begin{proof}[Proof of \Cref{lem:emp_proc}]
	By \Cref{lem:norm_atomic_equiv}, it suffices to show
	\[
	\sup_{\substack{\norm{u}_2 = \norm{v}_2 = 1\\ \norm{u}_0, \norm{v}_0 \leq s}} \ip{Z u}{v}
	\lesssim \sigma \sqrt{\frac{s \log (e p/s)}{n}} + \frac{M}{n^{1-c}} \parens*{s \log \frac{ep}{s}}^{\eta + 1}
	\]
	where, again, $Z = \frac{1}{n} \sum_i G_i$.
	
	We first consider the random variable $\ip{Z u}{v}$ for fixed unit-norm $u$ and $v$.
	We have
	\begin{align*}
		\ip{Zu}{v}
		&= \frac{1}{n} \sum_{i=1}^n \ip{G_i u}{v}.
	\end{align*}
	This is the sum of independent copies of the zero-mean random variable $\ip{G u}{v}$.
	By assumption,
	\[
	\E \ip{G u}{v}^2 \leq \sigma^2
	\]
	and, for $\alpha \geq 3$,
	\[
	\norm{\ip{G u}{v}}_\alpha \leq M \alpha^{\eta+1}.
	\]
	Then, by \cite[Theorem 3.1]{Rio2017}, for any $\delta > 0$, with probability at least $1 - \delta$,
	\[
	\ip{Z u}{v} \lesssim \sigma \sqrt{\frac{\log \delta^{-1}}{n}} + \frac{M \alpha^{\eta + 1}}{n^{1 - 1/\alpha}} \delta^{-1/\alpha}.
	\]		
	We then use a covering argument similar to that in \cite{Baraniuk2008}.
	Let $J_1$ and $J_2$ be any two subspaces of $s$-sparse vectors in $\R^p$.
	The unit sphere $S_{J_i}$ in $J_i$ can be covered within a resolution of $1/4$ by at most $9^s$ points (\cite[Corollary 4.2.13]{Vershynin2018}, for example).
	Let $\scrN_{J_1}, \scrN_{J_2}$ be optimal $1/4$-covering sets.
	For each $x \in S_{J_i}$, let $n_i(x)$ be the closest point in $\scrN_{J_i}$.
	Then
	\begin{align*}
		\sup_{\substack{u \in S_{J_1} \\ v \in S_{J_2}}} \ip{Zu}{v}
		&= \sup_{\substack{u \in S_{J_1} \\ v \in S_{J_2}}} \ip{Z n_1(u)}{n_2(v)} + \ip{Z (u - n_1(u))}{v} + \ip{Z n_1(u)}{v - n_2(v)} \\
		&\leq \max_{\substack{u \in \scrN_{J_1} \\ v \in \scrN_{J_2}}} \ip{Z u}{v} + \frac{1}{2} \sup_{\substack{u \in S_{J_1} \\ v \in S_{J_2}}} \ip{Zu}{v},
	\end{align*}
	so
	\[
	\sup_{\substack{u \in S_{J_1} \\ v \in S_{J_2}}} \ip{Zu}{v} \leq 2 \max_{\substack{u \in \scrN_{J_1} \\ v \in \scrN_{J_2}}} \ip{Z u}{v}.
	\]
	Let
	\[\scrN = \union_{\text{$s$-sparse $J_1,J_2$}} \scrN_{J_1} \times \scrN_{J_2}.\]
	Clearly,
	\begin{align*}
		\sup_{\substack{\norm{u}_2 = \norm{v}_2 = 1\\ \norm{u}_0, \norm{v}_0 \leq s}} \ip{Zu}{v}
		&= \sup_{\text{$s$-sparse $J_1, J_2$}}\ \sup_{\substack{u \in S_{J_1} \\ v \in S_{J_2}}} \ip{Zu}{v} \\
		&\leq 2 \max_{(u, v) \in \scrN}~\ip{Zu}{v}.
	\end{align*}
	There are $\binom{p}{s} \leq \parens*{\frac{ep}{s}}^s$ $s$-sparse subspaces of $\R^p$,
	so $\abs{\scrN} \leq \parens*{9^s \parens*{\frac{ep}{s}}^s}^2$.
	
	By a union bound and substituting $\delta$ above with $\delta / \abs{\scrN}$, we then have, for any $\delta > 0$, with probability at least $1 - \delta$,
	\[
	\sup_{\substack{\norm{u}_2 = \norm{v}_2 = 1\\ \norm{u}_0, \norm{v}_0 \leq s}} \ip{Zu}{v}
	\lesssim  \sigma \sqrt{\frac{s \log (e p/s)}{n} + \frac{\log \delta^{-1}}{n}} + \frac{M \alpha^{\eta + 1}}{n^{1 - 1/\alpha}} \parens*{\frac{e p}{s}}^{2 s/\alpha} \delta^{-1/\alpha}.
	\]
	Taking $\delta = e^{-s} (s/p)^s$ and $\alpha \approx s \log \frac{C p}{s}$,
	we get, with probability at least $1 -  e^{-s} (s/p)^s$,
	\[
	\sup_{\substack{\norm{u}_2 = \norm{v}_2 = 1\\ \norm{u}_0, \norm{v}_0 \leq s}} \ip{Zu}{v}
	\lesssim \sigma \sqrt{\frac{s \log (e p/s)}{n}} + \frac{M}{n^{1-c}} \parens*{s \log \frac{ep}{s}}^{\eta + 1}.
	\] 
\end{proof}

We will need the following variant of \Cref{lem:emp_proc} for both the sparse PCA results and our restricted lower isometry lemma:
\begin{lemma}
	\label{lem:emp_proc_exp}
	Let $G_1, \dots, G_n$ be i.i.d.\ copies of a random matrix $G \in \R^{p \times p}$,
	where, for all $u, v \in \R^p$, $\ip{Gu}{v}$ has zero mean,
	\[
		\E \ip{G u}{v}^2 \lesssim \norm{u}_2^2 \norm{v}_2^2
	\]
	and $\ip{G u}{v}$ is sub-exponential in the sense that $\norm{\ip{G u}{v}}_\alpha \lesssim \alpha \norm{u}_2 \norm{v}_2$ for all $\alpha \geq 2$.
	
	Let
	\[
	Z = \frac{1}{n} \sum_{i=1}^n G_i
	\]
	For any integer $s \geq 1$, with probability at least $1-e^{-s}(s/p)^s$,
	\[
		\sup_{\normMixeds{A} \leq 1} \ipHS{Z}{A}
		\lesssim \sqrt{\frac{s \log (e p/s)}{n}} + \frac{s \log (e p/s)}{n}.
	\]
	Furthermore, for $n \gtrsim s \log \frac{ep}{s}$,
	\[
		\E \sup_{\normMixeds{A} \leq 1} \ipHS{Z}{A}
		\lesssim \sqrt{\frac{s \log (e p/s)}{n}}.
	\]
\end{lemma}
We omit the proof, as it is nearly identical to the proof of \Cref{lem:emp_proc}.
We simply replace the Fuk-Nagaev inequality with a Bernstein inequality.
With this, we can prove our restricted lower isometry lemma:
\begin{proof}[Proof of \Cref{lem:rsc}]
	If $X = x \otimes x$,
	by a straightforward calculation, for any $p \times p$ matrix $A$,
	\[
	\E \ipHS{X}{A}^2 = \sum_{i \neq j} A_{ii}A_{jj} \E(x^{(i)})^2 (x^{(j)})^2
	+ 2 \sum_{i \neq j} A_{ij}^2 \E(x^{(i)})^2(x^{(j)})^2
	+ \sum_i A_{ii}^2 \E (x^{(i)})^4.
	\]
	Using the facts that $\E (x^{(i)})^2 = 1$ for each $i$ and $x^{(i)}$ and $x^{(j)}$ are independent when $i \neq j$,
	we have
	\begin{align*}
		\E \ipHS{X}{A}^2 &= \sum_{i,j} A_{ii} A_{jj} +2 \sum_{i \neq j} A_{ij}^2 + \sum_i A_{ii}^2 (\E (x^{(i)})^4 - 1) \\
		&\geq (\tr A)^2 + \min\{2, \E (x^1)^4 - 1 \} \norm{A}_F^2 \\
		&\gtrsim \norm {A}_F^2.
	\end{align*}
	The last inequality uses the assumption that $\E (x^{(1)})^4 > 1$.
	
	By the Hanson-Wright inequality for sub-Gaussian vectors \cite{Rudelson2013},
	we have
	\[
	\E (\ipHS{X}{A}^2 - \E \ipHS{X}{A}^2)^2 \lesssim \norm{A}_F^4,
	\]
	so $\E \ipHS{X}{A}^4 \lesssim (\E \ipHS{X}{A}^2)^2$.
	By the Paley-Zygmund inequality, we then have, for some $c_1,c_2 > 0$,
	\[
	\inf_{A \in \R^{p \times p}} \P(\abs{\ipHS{X}{A}} \geq c_1 \norm{A}_F ) \geq c_2.
	\]
	The remainder of the proof is a small-ball argument (\cite{Mendelson2015}; see also \cite{Tropp2015a} for an excellent introduction).
	
	Let
	\[
	S = \{ A \in \R^{p \times p} : \norm{A}_F = 1;\ \normMixeds{A} \leq C \}.
	\]
	We will prove that
	\[
	\inf_{A \in S} \frac{1}{n} \sum_{i=1}^n \ipHS{X_i}{A}^2 \geq c
	\]
	with high probability for some constant $c > 0$.
	
	By \cite[Proposition 5.1]{Tropp2015a}, for any $t > 0$, we have, with probability at least $1 - e^{-t^2/2}$,
	\[
	\inf_{A \in S} \sqrt{\frac{1}{n} \sum_{i=1}^n \ipHS{X_i}{A}^2}
	\gtrsim c_1 c_2 - 2 \E \sup_{A \in S} \parens*{ \frac{1}{n} \sum_{i=1}^n \varepsilon_i \ipHS{X_i}{A} } - \frac{1}{\sqrt{n}} c_1 t,
	\]
	where $\varepsilon_1, \dots, \varepsilon_n$ are i.i.d.\ Rademacher random variables independent of everything else.
	
	Set $Z = \frac{1}{n} \sum_{i=1}^n \varepsilon_i X_i$, and note that $G_i = \varepsilon_i X_i$, $i = 1, \dots, n$, satisfy the requirements of \Cref{lem:emp_proc_exp}.
	Then
	\[
		\E \sup_{A \in S} \ipHS{Z}{A} \lesssim C \underline{}\sqrt{\frac{s \log (ep/s)}{n} }.
	\]
	Choosing $n$ large enough and $t = \sqrt{2 b n}$ for small enough $b > 0$ completes the proof.
	
\end{proof}

\section{Proof of sparse PCA error bound}
\newcommand{\projbperp}{\scrP_{(v_1 \otimes v_1)^\perp}}
\label{app:pca_proof}
\begin{proof}[Proof of \Cref{thm:pca}]
	By a similar argument to that in the proof of \Cref{thm:pr} in \Cref{sec:pr_proof},
	the solution $\Phat$ to \eqref{eq:pca_opt} satisfies
	\[
	\ipHS{\Sigmahat}{-H} \leq \lambda \ipHS{W}{-H}
	\]
	for $H = \Phat - P_1$ and any $W \in \partial \normMixeds{P_1}$.
	Choosing $W$ according to \Cref{lem:subgrad_ineq} (as in the proof of \Cref{thm:pr}),
	we obtain
	\[
	\ipHS{\Sigmahat}{H}
	\geq \lambda \parens*{ \frac{1}{10} \normMixeds{H} - 5 \norm{H}_F }.
	\]
	
	We first consider the difference between $\ipHS{\Sigmahat}{H}$ and $\ipHS{\Sigma}{H}$.
	Since the distribution of $\Sigmahat$ is independent of $\mu$,
	we assume, without loss of generality, that $\mu = 0$.
	We write $x_i = \Sigma^{1/2} z_i$, where $z_i \sim \normaldist(0, I_p)$,
	and $\Sigma^{1/2} = \sqrt{\sigma_1} P_1 + \Sigma_2^{1/2}$.
	We therefore want to bound
	\[
		\ipHS{\Sigmahat - \Sigma}{H} = \ipHS{\Sigma^{1/2} (Z - I_p - \zbr \otimes \zbr) \Sigma^{1/2}}{H},
	\]
	where $Z = \frac{1}{n} \sum_{i=1}^n z_i \otimes z_i$
	and $\zbr = \frac{1}{n} \sum_{i=1}^n z_i$.
	
	Let $H^\perp$ denote the component of $H$ orthogonal (in Hilbert-Schmidt inner product) to $P_1$.
	We have 
	\[
		H = \ipHS{H}{P_1} P_1 + H^\perp.
	\]
	First, for all $t \leq n$, with probability at least $1 - e^{-t}$,
	\begin{align*}
		\abs*{\ipHS{\Sigmahat - \Sigma}{P_1}}
		&= \sigma_1 \abs*{\frac{1}{n} \sum_{i=1}^n (\ip{z_i}{v_1}^2 - 1) - \ip{\zbr}{v_1}^2} \\
		&\leq \sigma_1 \abs*{\frac{1}{n} \sum_{i=1}^n (\ip{z_i}{v_1}^2 - 1) } + \ip{\zbr}{v_1}^2 \\
		&\lesssim \sigma_1 \parens*{\sqrt{\frac{t}{n}} + \frac{t}{n}} \\
		&\lesssim \sigma_1 \sqrt{\frac{t}{n}},
	\end{align*}
	where the second-to-last inequality follows from applying a Bernstein inequality to the sum and an ordinary Gaussian tail bound to the $\normaldist(0, 1/n)$ random variable $\ip{\zbr}{v_1}$.
	
	To analyze the remainder, denote the portion of $\Sigmahat$ orthogonal to $P_1$ as
	\begin{align*}
		\Sigmahat^\perp
		&= \Sigmahat - \ipHS{\Sigmahat}{P_1} P_1 \\
		&= \frac{1}{n} \sum_{i=1}^n \parens*{ \sqrt{\sigma_1} \ip{z_i}{v_1} ( v_1 \otimes (\Sigma_2^{1/2} z_i) + (\Sigma_2^{1/2} z_i) \otimes v_1 ) + (\Sigma_2^{1/2} z_i)^{\otimes 2} } - (\Sigma_2^{1/2} \zbr)^{\otimes 2}.
	\end{align*}
	Note that for each $i$, $\ip{v_1}{z_i}$ is independent of $\Sigma_2^{1/2} z_i$.
	By \Cref{lem:emp_proc_exp}, with probability at least $1 - 2 e^{-s} (s/p)^s$,
	\begin{align*}
		\sup_{\normMixeds{A} \leq 1} \ipHS{\Sigmahat^\perp + (\Sigma_2^{1/2} \zbr)^{\otimes 2} - \Sigma_2}{A}
		&\leq 2 \sup_{\normMixeds{A} \leq 1}~\ipHS*{ \frac{1}{n} \sum_{i=1}^n \sqrt{\sigma_1} \ip{z_i}{v_1} v_1 \otimes (\Sigma_2^{1/2} z_i) }{A} \\
		&\qquad+ \sup_{\normMixeds{A} \leq 1}~\ipHS*{ \frac{1}{n} \sum_{i=1}^n (\Sigma_2^{1/2} z_i)^{\otimes 2} - \Sigma_2}{A} \\
		&\lesssim (\sqrt{\sigma_1 \sigma_2} + \sigma_2) \parens*{ \sqrt{\frac{s \log (ep/s)}{n} } + \frac{s \log (ep/s)}{n} } \\
		&\lesssim \sqrt{\sigma_1 \sigma_2} \sqrt{\frac{s \log (ep/s)}{n} }.
	\end{align*}
	\Cref{lem:emp_proc_exp} also gives, with probability at least $1 - e^{-s} (s/p)^s$,
	\begin{align*}
		\sup_{\normMixeds{A} \leq 1}~\ipHS{(\Sigma_2^{1/2} \zbr)^{\otimes 2}}{A}
		&\leq \sup_{\normMixeds{A} \leq 1}~\ipHS{(\Sigma_2^{1/2} \zbr)^{\otimes 2} - \E (\Sigma_2^{1/2} \zbr)^{\otimes 2} }{A} + \sup_{\normMixeds{A} \leq 1}~\ipHS{\E (\Sigma_2^{1/2} \zbr)^{\otimes 2}}{A} \\
		&\lesssim \sigma_2 \sqrt{\frac{s \log (ep/s)}{n} }.
	\end{align*}
	Therefore,
	\[
		\sup_{\normMixeds{A} \leq 1}~\ipHS{\Sigmahat^\perp - \Sigma_2}{A}
		\lesssim \sqrt{\sigma_1 \sigma_2} \sqrt{\frac{s \log (ep/s)}{n} }
	\]
	with probability at least $1 - 3 e^{-s} (s/p)^s$.

	Let $\lambda$ be chosen with a large enough constant to ensure that on this event,
	\[
		\sup_{\normMixeds{A} \leq 1} \ipHS{\Sigmahat^\perp - \Sigma_2}{A} \leq \frac{\lambda}{10}.
	\]
	Then
	\[
		\abs{ \ipHS{\Sigmahat^\perp - \Sigma_2}{H} } \leq \frac{\lambda}{10} \normMixeds{H}.
	\]
	
	We then have
	\begin{align*}
		\sigma_1 \ipHS{P_1}{H} +\ipHS{\Sigma_2}{H}
		&= \ipHS{\Sigma}{H} \\
		&= \ipHS{\Sigmahat}{H} + \ipHS{\Sigma - \Sigmahat}{H} \\
		&\geq \lambda\parens*{\frac{1}{10} \normMixeds{H} - 5 \norm{H}_F} - \sigma_1 \sqrt{\frac{t}{n}} \abs{\ipHS{H}{P_1}} - \frac{\lambda}{10} \normMixeds{H} \\
		&= - 5\lambda \norm{H}_F - \sigma_1 \sqrt{\frac{t}{n}} \abs{\ipHS{H}{P_1}}.
	\end{align*}
	Note that
	\[
		\ipHS{H}{P_1} = \ipHS{\Phat}{P_1} - 1 \leq 0,
	\]
	and $\ipHS{\Sigma_2}{H} = \ipHS{\Sigma_2}{\Phat}$,
	so
	\begin{align*}
		\sigma_1 \parens*{ 1 - \sqrt{\frac{t}{n}} } (\ipHS{\Phat}{P_1} - 1) + \ipHS{\Phat}{\Sigma_2} \gtrsim - \lambda \norm{H}_F.
	\end{align*}
	
	Note that $\ipHS{\Phat}{\Sigma_2} \leq \sigma_2 \norm{\projTp(\Phat)}_*$,
	where $T^\perp$ is (similarly to before) the matrix subspace with rows and columns orthogonal to $v_1$.
	Note that $1 \geq \norm{\Phat}_* \geq \ipHS{\Phat}{P_1} + \norm{\projTp(\Phat)}_*$,
	so $\ipHS{\Phat}{\Sigma_2} \leq \sigma_2(1 - \ipHS{\Phat}{P_1})$.
	
	Combining this with the previous inequality and requiring $n \gtrsim \parens*{\frac{\sigma_1}{\sigma_1 - \sigma_2}}^2 t$,
	we have
	\[
		\parens*{ \sigma_1 - \sigma_2 } (1 - \ipHS{\Phat}{P_1})
		\lesssim \parens*{ \sigma_1 \parens*{1 - \sqrt{\frac{t}{n}}} - \sigma_2 } (1 - \ipHS{\Phat}{P_1})
		\lesssim \lambda \norm{H}_F.
	\]
	To bound $\norm{H}_F$, note that we can write
	\[
		\Phat = a v_1 \otimes v_1 + v_1 \otimes u + w \otimes v_1 + \scrP_{T^\perp}(\Phat),
	\]
	where $a = \ipHS{\Phat}{P_1}$ and $u,w \perp v_1$.
	Then
	\[
		1 \geq \norm{\Phat}_*^2 \geq \norm{\Phat}_F^2 = a^2 + \norm{u}_2^2 + \norm{w}_2^2 + \norm{\scrP_{T^\perp}(\Phat)}_F^2,
	\]
	and therefore
	\begin{align*}
		\norm{H}_F^2
		&= (1 - a)^2 + \norm{u}_2^2 + \norm{w}_2^2 + \norm{\scrP_{T^\perp}(\Phat)}_F^2 \\
		&\leq (1-a)^2 + 1 - a^2 \\
		&= 2(1 - a) \\
		&= 2(1 - \ipHS{\Phat}{P_1}).
	\end{align*}
	From this, we have $(\sigma_1 - \sigma_2) \norm{H}_F^2 \lesssim \lambda \norm{H}_F$,
	from which the result immediately follows.
\end{proof}

\section{Proof of Poisson variance/moment bounds}
\label{app:poisson}
If $x$ satisfies \Cref{assump:spr_meas} and, conditioned on $x$, $y \sim \poissondist(\ip{x}{\beta^*}^2)$,
then, for unit-norm $u \in \R^p$,
\begin{align*}
	\E \xi^2 \ip{x}{u}^4
	&= \E \brackets*{ \E[\xi^2 \given x] \ip{x}{u}^4 } \\
	&= \E \ip{x}{\beta^*}^2 \ip{x}{u}^4 \\
	&\lesssim \norm{\beta^*}_2^2.
\end{align*}
Also,
\begin{align*}
	\norm{\xi \ip{x}{u}^2}_\alpha
	&= \parens*{ \E\abs*{ \xi \ip{x}{u}^2 }^\alpha }^{1/\alpha} \\
	&= \parens*{\E \brackets*{ \E[ \abs{\xi}^\alpha \given x ] \abs{\ip{x}{u}}^{2\alpha} }}^{1/\alpha} \\
	&\lesssim \sqrt{\alpha} \parens*{\E  \abs{\ip{x}{\beta^*}}^\alpha \abs{\ip{x}{u}}^{2\alpha}}^{1/\alpha} + \alpha \norm{\ip{x}{u}^2}_\alpha \\
	&\lesssim \alpha^2( \norm{\beta^*}_2 + 1 ),
\end{align*}
where the first inequality uses the standard Poisson centered moment bound
\[
	\norm{Z - \E Z}_\alpha \lesssim \sqrt{\alpha \lambda} + \alpha 
\]
if $Z \sim \poissondist(\lambda)$.

	\printbibliography[heading=bibintoc]
\end{document}